\documentclass[11pt]{amsart}
\usepackage[pagewise]{lineno}
\usepackage{amsthm, amsmath,amsfonts,amssymb,euscript,hyperref,graphics,color,slashed,cite}
\usepackage{graphicx}
\usepackage{comment}
\usepackage{import}
\usepackage{tikz}
\usepackage{latexsym}
\usepackage{upgreek}
\usepackage[arrow,matrix]{xy}
\usepackage{tabularx}
\usepackage{cases}
\usetikzlibrary{positioning}
\usetikzlibrary{cd}
\usetikzlibrary{decorations.pathreplacing}
\newtheorem*{Main Theorem}{Main Theorem}
\newtheorem*{Main Proposition}{Main Proposition}

\newtheorem{theorem}{Theorem}[section]
\newtheorem{lemma}[theorem]{Lemma}
\newtheorem{proposition}[theorem]{Proposition}

\newtheorem{remark}[theorem]{Remark}

\numberwithin{equation}{section}

\begin{document}

\title {Radiation fields for semilinear Dirac equations with spinor null forms}

\author{Jin Jia}
\address{School of Mathematics, Hunan University, Changsha, China}
\email{jiajin2023@hnu.edu.cn}

\author{Jiong-Yue LI}
\address{Department of Mathematics, Sun Yat-sen University, Guangzhou, China}
\email{lijiongyue@mail.sysu.edu.cn}


\begin{abstract}
	In this paper, we investigate the scattering theory of half-spin waves through the use of radiation fields. We define the radiation fields for semilinear Dirac equations with spinor null forms and establish a nonlinear isomorphism between the weighted energy space of initial data and the weighted energy space of radiation fields. The proof relies on a detailed examination of the linear Dirac radiation fields and the use of a suitable functional framework.
\end{abstract}

\maketitle

\section{Introduction}
The Dirac equation, formulated by British physicist Paul Dirac in 1928, is a fundamental equation in quantum mechanics that describes the behavior of fermions, such as electrons, which are particles with half spin \cite{Dirac P A M}. The equation is significant because it successfully merges quantum mechanics with special relativity and predicts the existence of antimatter. The Dirac equation for a free particle is given by:
\[
(i \gamma^\mu \partial_\mu - m) \psi = 0,
\]
where $m$ is the mass of the particle, $\partial_\mu$ with $\mu=0,1,2,3$ are partial derivatives with respect to spacetime coordinates, and Einstein summation convention is used.
Here $\psi: \mathbb{R}^{1+3}\rightarrow \mathbb{C}^{4}$ is a four-component spinor known as the Dirac spinor and is the wave function of the particle.
The Dirac equation is written using four gamma matrices, $\gamma^0, \gamma^1, \gamma^2, \gamma^3$, which satisfy specific anticommutation relations:
\[ 
      \{\gamma^\mu, \gamma^\nu\} := \gamma^\mu \gamma^\nu + \gamma^\nu \gamma^\mu = -2\eta^{\mu\nu}I, 
\] 
where $\eta^{\mu\nu}$ is the metric tensor of spacetime with signature typically chosen as $(-,+,+,+)$ which will be used to raise or lower indices, and $I$ is the $4\times4$ identity matrix.

The gamma matrices help to describe intrinsic spin and various symmetries in particle physics. There are multiple representations of the gamma matrices. The most commonly used are the Dirac representation (or standard representation) and the Weyl representation (or chiral representation). In the Weyl representation, the gamma matrices are typically written as: 
\begin{align*}
	\gamma^{0}=\begin{pmatrix}
		0& I  \\  I&0
	\end{pmatrix},\quad
	\gamma^{i}=\begin{pmatrix}
		0&\sigma^{i}\\ -\sigma^{i}&0
	\end{pmatrix}, \quad i=1,2,3,
\end{align*}
where $I$ is $2\times2$ identity matrix and $\sigma^{i}$ are Pauli matrices. The Pauli matrices are a set of three $2\times2$ complex matrices used in quantum mechanics:
\begin{align*}
	\sigma^{1}=\begin{pmatrix}
		0&1\\ 1&0
	\end{pmatrix}\quad
	\sigma^{2}=\begin{pmatrix}
		0&-i\\ i&0
	\end{pmatrix}\quad
	\sigma^{3}=\begin{pmatrix}
		1&0 \\ 0&-1
	\end{pmatrix},
\end{align*}

Shortly after the development of the linear Dirac equation, physicists began investigating potential nonlinear generalizations to explain particle interactions and other phenomena. In 1953, Werner Heisenberg \cite{heisenberg1984quantum} proposed a nonlinear spinor theory as an attempt to describe elementary particles and their interactions. His approach was one of the early significant efforts to introduce nonlinearity into the Dirac framework. In 1958, Walter Thirring \cite{Thirring Walter E} introduced a model that included a nonlinear self-interaction term for fermions. The Thirring model was originally formulated as a model in $(1 + 1)$ space-time dimensions and was characterized by the Lagrangian density
\begin{align*}
     \mathcal{L}=\bar{\psi}\big(i\sigma^{\mu}\partial_{\mu}-m\big)\psi-\frac{a}{2}\big(\bar{\psi}\sigma^{\mu}\psi\big)\big(\bar{\psi}\sigma_{\mu}\psi\big),
\end{align*}
where $\psi:\mathbb{R}^{1+1}\rightarrow\mathbb{C}^{2}$ is the spinor field, $\bar{\psi}=\psi^{*}\gamma^{0}$ is the Dirac adjoint spinor, $\psi^{*}$ is the conjugate transpose of $\psi$, $a$ is the coupling constant, 
$m$ is the mass and $\mu=0,1$ is the index. In 1970, a simplified nonlinear spinor model known as the Solar model  in $(3+1)$ space-time dimensions was introduced \cite{Soler M}, featuring a specific Lagrangian density
\begin{align*}
   \mathcal{L}=\bar{\psi}\big(i\gamma^{\mu}\partial_{\mu}-m\big)\psi+\frac{a}{2}\big(\bar{\psi}\psi\big)^{2},
\end{align*}
where $\psi:\mathbb{R}^{1+3}\rightarrow \mathbb{C}^{4}$ and $\gamma^{\mu}$ denotes the Dirac gamma matrices $\gamma^{\mu}$ with $\mu=0,1,2,3$. The mathematical characteristics of nonlinear Dirac equations, including the existence and uniqueness of solutions, stability, and long-term behavior, became areas of significant focus. Researchers devised methods to examine these properties in different contexts. 

In 1980s, a crucial breakthrough was achieved by Klainerman \cite{MR837683} and Christodoulou \cite{Christodoulou D},  through the introduction of the celebrated null conditions.  Specifically, for a nonlinear wave equation, if the quadratic part of the nonlinearity is a null form, the existence of small-data global solution was proved, which led to the proof of the global nonlinear stability of the Minkowski spacetime \cite{MR1316662}.

Nonlinear Dirac equations, including the Thirring and Soler models, exhibit invariance under Lorentz transformations. It is believed that these equations inherently possess special structures. In 2007,  D’Ancona, Foschi, and Selberg uncovered a null structure in the Dirac equation, which cannot be seen directly, but appeared in the Dirac-Klein-Gordon system after a duality argument \cite{DAnconaFoschiSelberg}. In 2014, Katayama and Kubo \cite{MR3362023} introduced a null condition related to the Dirac equation $D\phi=N(\phi,\phi)$ in \cite{MR3362023}:
\begin{align*}
     N\big(P(\omega)\phi,P(\omega)\phi\big)=0 \quad \text{for all}\ \omega\in\mathbb{S}^{2}, \ \phi\in\mathbb{C}^{4},
\end{align*}
where $N(\cdot,\cdot):\mathbb{C}^{4}\times\mathbb{C}^{4}\rightarrow \mathbb{C}^{4}$ is a bilinear form and
\begin{align*}
P(\omega):=\frac{1}{2}(I+\gamma^{0}\gamma^{i}\omega_{i}),
\end{align*}
is a projection operator on the spinor bundle. Their result gave an affirmative answer to the Tzvetkov's conjecture as proposed in \cite{MR1637692}. Tzvetkov conjectured that the nonlinear Dirac equation has a small-data global solution, provided that the nonlinearity takes such special forms as 
\begin{align*}
    N_{1}(\phi,\phi)e=\langle\gamma^{0}\phi, \phi\rangle e \quad\text{and}\quad  N_{2}(\phi, \phi)e=\langle\gamma^{0}\gamma^{5}\phi, \phi\rangle e,
\end{align*}
where $e\in\mathbb{C}^{4}$ is an arbitrary constant vector and $\langle\cdot, \cdot\rangle$ is the inner product on $\mathbb{C}^{4}$. The $4\times4$ matrix $\gamma^{5}:=-i\gamma^{0}\gamma^{1}\gamma^{2}\gamma^{3}$. It is important to note that $N_{1}$
also appears in the Soler model. Katayama-Kubo's null form is a slightly more general version of the above formulas $N_{1}$ and $N_{2}$. In 2021, Li and Zang proposed an alternative definition of the spinor null form using the Newman-Penrose formalism \cite{MR4184660}. Their definition reveals that the nonlinear terms of the Thirring model and the Soler model indeed satisfy the spinor null structure. By combining the peeling decay property of the solution with energy estimates, a small-data global solution can be found in Klainerman's weighted Sobolev space. For the cubic Dirac equation on $\mathbb{R}^{1+3}$, Bournaveas and Candy \cite{BournaveasCandy} used Fourier analysis to study global well poseness in critical energy space, see \cite{BejenaruHerr,GavrusOh} for other space time dimesions.

The aim of this paper is study the scattering theory of the nonlinear Dirac equations with spinor null forms by analyzing the radiation field of the spinor waves. Combing the vector field method and the energy method, we will prove the existence of the nonlinear radiation field and establish a local isomorphism between a weighted energy space of the initial data and a weighted energy space of the radiation field.
                                                                                                                                                                                                                                                                                                                                                                 
\subsection{Motivation and previous works}
Scattering theory is an old discipline. In 1910, Rutherford conducted the landmark hydrogen atom scattering experiment and provided significant evidence for the nuclear model of the atom \cite{Rutherford1911TheSO}. 
Peter D. Lax and Ralph S. Phillips made significant contributions to the rigorous mathematical formulation of scattering theory for hyperbolic differential equations \cite{MR0217440}. Their approach is characterized by the use of functional analysis and semigroup theory to study the scattering process and is known as the Lax-Phillips scattering theorty. In 1960s, Friedlander introduced the concept of the radiation field . The radiation field captures the behavior of the wave at large distances and can be defined in terms of the asymptotic expansion of the solution along the null direction \cite{MR344687,MR583989,MR1846782}.

In the scattering theory, once we formulate a suitable notion of scattering data, we should ask the following questions \cite{Reed} : 
\begin{itemize}
     \item[(i)] {\it Existence of the scattering states:} For any given scattering data, does there exist a global solution to the equation that corresponds to this scattering data?  
     \item[(ii)]{\it Uniqueness of the scattering states:} If two global solutions correspond to the same scattering data, must they be identical?
     \item[(iii)]{\it Asymptotic completeness:} Do the solutions determined by the scattering data encompass all the global solutions to the equation?\end{itemize}
The literature on scattering theory for linear wave equations is relatively extensive. For wave-type equations and the Teukolsky equation on the exterior of the Schwarzschild black hole, see \cite{BaskinWang, Masaood, MR1317184, MR2047861}. For the wave equation on the exterior of the subextremal Kerr black hole, refer to \cite{DafermosRodnianskiShlapentokh}. For the wave equation on the extremal Reissner–Nordström black hole, see \cite{AngelopoulosAretakisGajic}. Additionally, for the wave equation on the Oppenheimer–Snyder background, refer to \cite{Alford}, and for the massless Dirac–Coulomb system, see \cite{BaskinBoothGell}.
Nonlinear scattering theory is less developed compared to the linear theory.  In the context of the Maxwell–Klein–Gordon system, \cite{CandyKauffmanLinblad, He L, ChenLinblad} gave affirmative answer to \textit{asymptotic completeness} problem. In the small data regime, we refer to \cite{Yu D1} for wave equations satisfying a class of weak null conditions, and to \cite{Yu D2} for compressible Euler equations.  On the \textit{uniqueness of scattering states}, see \cite{Zha D} for elastic wave equations, \cite{MR4223342} for nonlinear Alfvén waves, and \cite{JiaYu} for perturbations of Minkowski space.

The aforementioned works on \textit{asymptotic completeness} are based on the approximation method developed by Lindblad and Schlue \cite{2017arXiv171100822L}, which results in a loss of derivatives during the limiting process. Our work is inspired by Baskin and Sá Barreto \cite{MR2169870}, who developed a functional framework by treating the nonlinear term as a source term to study the radiation field for the critical semilinear wave equation on $\mathbb{R}^{1+3}$. This framework provides a possibility to view the nonlinear scattering map as a nonlinear operator, enabling the application of the inverse function theorem in an appropriate Banach space for radiation fields. The choice of such a Banach space must be compatible with both nonlinear estimates and the radiation field map. Our main innovation lies in the construction of these spaces. In this paper, we construct weighted Sobolev spaces using appropriate vector fields and establish a local isomorphism theorem between initial data and scattering data, thus avoiding the problem of derivative loss. After finishing this work, we learned that Li \cite{Li M} has applied similar methods for nonlinear Alfvén waves.

To motivate the concept of the radiation field for the Dirac equation, we will provide a brief overview of scattering theory for free wave equations. For wave type equations, the far field patterns can be effected represented by means of radiation field, which was raised by Friedlander in a series of papers\cite{MR583989,MR344687,MR1846782}. To discuss the main idea, let's focus on the following Cauchy problem
of wave equation with compactly supported source term and initial data in $\mathbb{R}^{1+3}$,
\begin{equation}\label{dirac: wave equation}
\begin{cases}
\Box u(t,x)=G(t,x),\\
(u,\partial_{t}u)|_{t=0}=(f(x),g(x)),
\end{cases}
\end{equation}
where $u(t,x):\mathbb{R}^{1+3}\rightarrow \mathbb{R}$ is a function, $\Box=-\partial_{t}^{2}+\sum_{i}^{3}\partial_{i}^{2}$ is the wave operator, $(f, g)\in C_{0}^{\infty}(\mathbb{R}^{3})\times C_{0}^{\infty}(\mathbb{R}^{3})$ is the initial data, and $G(t,x)\in C_{0}^{\infty}(\mathbb{R}^{1+3})$ denotes the source term. According to Friedlander \cite{MR583989}, the linear wave $u(t,x)$ have the following properties:

\medskip

In the region $\{(t,x):t\geq 0, |x|\geq 1\}$. The solution $u(t,x)$ has the following expansion in terms of $\frac{1}{|x|}$.
\begin{align*}
u(t,x)=\sum_{k=1}^{\infty}\frac{u_{k}(t-|x|,\frac{x}{|x|})}{|x|^{k}},
\end{align*}
where $u_{k}(s,\omega)\in C_{0}^{\infty}(\mathbb{R}\times\mathbb{S}^{2})$ are smooth functions and the leading term $u_{1}(s,\omega)$ dominates the forward far field pattern of the solution $u(t,x)$. {\it The future radiation field associated with Eq. \eqref{dirac: wave equation} is defined to be the following limit  }
\begin{align*}
\mathcal{F}_{wave}^{+}(f,g,G)(s,\omega):=\lim_{r\rightarrow +\infty} r\partial_{t}u(r+s,r\omega)=\partial_{s}u_{1}(s,\omega).
\end{align*}
which is a smooth function in $C_{0}^{\infty}(\mathbb{R}\times\mathbb{S}^{2})$.

\medskip

In the region $\{(t,x):t\leq 0, |x|\geq 1\}$,  the solution $u(t,x)$ has the following expansion in terms of $\frac{1}{|x|}$
\begin{align*}
u(t,x)=\sum_{k=1}^{\infty}\frac{v_{k}(t+|x|,\frac{x}{|x|})}{|x|^{k}},
\end{align*}
where $v_{k}(s,\omega)\in C_{0}^{\infty}(\mathbb{R}\times\mathbb{S}^{2})$ are smooth functions and the leading term $v_{1}(s,\omega)$ dominates the backward far field pattern of the solution $u(t,x)$. {\it The past radiation field associated with Eq. \eqref{dirac: wave equation} is defined to be the following limit  }
\begin{align*}
\mathcal{F}^{-}_{wave}(f,g,G)(s,\omega):=\lim_{r\rightarrow \infty} r\partial_{t}u(s-r,r\omega)=\partial_{s}v_{1}(s,\omega),
\end{align*}
which is also a smooth function in $C_{0}^{\infty}(\mathbb{R}\times\mathbb{S}^{2})$.
Moreover, if the source term in Eq. \eqref{dirac: wave equation} vanishes ($G\equiv0$), 
\begin{align*}
\frac{1}{2}\int_{\mathbb{R}^{3}}|\nabla f(x)|^{2}+|g(x)|^{2}dx=\int_{\mathbb{R}\times\mathbb{S}^{2}} |\partial_{s}u_{1}(s,\omega)|^{2} d\omega ds,\\
\frac{1}{2}\int_{\mathbb{R}^{3}}|\nabla f(x)|^{2}+|g(x)|^{2}dx=\int_{\mathbb{R}\times\mathbb{S}^{2}} |\partial_{s}v_{1}(s,\omega)|^{2} d\omega ds,
\end{align*}
where $d\omega$ denotes the standard sphere measure on $\mathbb{S}^{2}$. 
Furthermore, the forward radiation field map:
\begin{align*}
\mathcal{F}^{+}_{wave}: \dot{H}^{1}(\mathbb{R}^{3})\times &L^{2}(\mathbb{R}^{3})\rightarrow L^{2}(\mathbb{R}\times\mathbb{S}^{2}), \\
&(f,g)\mapsto \mathcal{F}^{+}_{wave}(f, g,0)
\end{align*}
is an isomorphism. Similarly, {\it the past radiation field map:}
\begin{align*}
\mathcal{F}^{-}_{wave}: \dot{H}^{1}(\mathbb{R}^{3})\times &L^{2}(\mathbb{R}^{3})\rightarrow L^{2}(\mathbb{R}\times\mathbb{S}^{2}), \\
&(f,g)\mapsto \mathcal{F}^{-}_{wave}(f, g,0)
\end{align*}
is also an isomorphism. Then the scattering operator associated to Eq. \eqref{dirac: wave equation} with vanishing source term can be defined as 
\begin{align*}
\mathcal{S}:=(\mathcal{F}^{+}_{wave})\circ(\mathcal{F}^{-}_{wave})^{-1}: L^{2}(\mathbb{R}\times\mathbb{S}^{2})\rightarrow L^{2}(\mathbb{R}\times\mathbb{S}^{2}).
\end{align*}
Thus the "existence, uniqueness and completeness" questions are well answered.

\bigskip

In this paper, our objective is to investigate the radiation fields of spinor waves. We will begin by examining the linear Dirac equation and associated linear radiation field. Based on the linear theory, we will explore the radiation fields of the semi-linear Dirac equations with spinor null forms. Consider the Cauchy problem of the Dirac equation in Minkowski space $\mathbb{R}^{1+3}$ 
\begin{equation}\label{dirac: dirac equation}
\begin{cases}
\mathcal{D}\phi=\Phi,\\
\phi(0,x)=\phi_{0}(x).
\end{cases}
\end{equation}
where the initial data $\phi_{0}(x)\in C_{0}^{\infty}(\mathbb{R}^{3},\mathbb{C}^{4})$ and $\Phi(t,x)\in C_{0}^{\infty}(\mathbb{R}^{1+3},\mathbb{C}^{4})$ is the source term.
We define the {\it future radiation field} of a spinor wave to Eq.(\ref{dirac: dirac equation}) as:
\begin{align*}
\mathcal{F}^{+}_{Dirac}(\phi_{0},\Phi)(s,\omega):=\lim_{r\rightarrow+\infty}r\phi(s+r,r\omega),
\end{align*}
and the {\it past radiation field} of this spinor wave as:
\begin{align*}
\mathcal{F}^{-}_{Dirac}(\phi_{0},\Phi)(s,\omega):=\lim_{r\rightarrow+\infty}r\phi(s-r,r\omega),
\end{align*}
Here $\phi_{0}$ is the initial data and $\Phi$ is the source term. 

The definition suggests that both $\mathcal{F}^{+}_{Dirac}(\phi_{0},\Phi)(s,\omega)$ and $\mathcal{F}^{-}_{Dirac}(\phi_{0},\Phi)(s,\omega)$ are in $C_{0}^{\infty}(\mathbb{R}\times\mathbb{S}^{2},\mathbb{C}^{4})$. However, in \cite{MR3362023,MR4184660} it is observed that $P(-\omega)\phi$ enjoys better peeling estimate when $t\geq 0$ and $P(\omega)\phi$ enjoys better peeling estimate when $t\leq 0$, which motivates us to define the spinor spaces
\begin{align*}
E^{+}:&=\{\psi\in C_{0}^{\infty}(\mathbb{R}\times\mathbb{S}^{2},\mathbb{C}^{4}):P(-\omega)\psi=0 \},\\
E^{-}:&=\{\psi\in C_{0}^{\infty}(\mathbb{R}\times\mathbb{S}^{2},\mathbb{C}^{4}):P(\omega)\psi=0 \},
\end{align*}
It will be shown that $\mathcal{F}^{+}_{Dirac}(\phi_{0},\Phi)$ is an element in $E^{+}$ and $\mathcal{F}^{-}_{Dirac}(\phi_{0},\Phi)$ is an element in $E^{-}$. 

In this paper, we will study radiation fields of semi-linear Dirac equations with spinor null forms, which are general equations of the Thirring model and the Soler model. We now consider the Cauchy problem for the semi-linear Dirac equation
\begin{equation}\label{dirac: semi-linear dirac equation}
\begin{cases}
\mathcal{D}\phi=N(\phi,\phi),\\
\phi(0,x)=\phi_{0}(x).
\end{cases}
\end{equation}
where $N(\phi,\phi):=\left\langle \gamma^{0}\phi,\phi \right\rangle e_{1}+\left\langle \gamma^{0}\gamma^{5}\phi,\phi \right\rangle e_{2}$ with $e_{1},e_{2}\in\mathbb{C}^{4}$ are complex constant vectors in $\mathbb{C}^{4}$.

\begin{remark}
In the framework of the vector field methods, cubic and higher order terms can also be handled. For simplicity, we only consider quadratic nonlinearities in this paper. 
\end{remark}

As in the case of wave equation, we can establish an isomorphism between the $L^{2}$ energy space of the initial data to the $L^2$ energy space of the radiation field. However, this is insufficient for addressing the nonlinear problem. Next we will extend this linear radiation field map to an isomorphism between a weighted energy space of the initial data and a weighted energy space of the radiation field.

\medskip

{\it The weighted energy space of the initial data $\mathcal{X}$}: We hope this weighted energy space will be suitable for a bootstrap argument to construct a global solution for the Cauchy problem Eq. \eqref{dirac: semi-linear dirac equation}. More precisely, the global solution $\phi$ satisfies
\begin{align}
\|\phi(t,\cdot)\|_{\mathcal{X}}\lesssim \|\phi_{0}\|_{\mathcal{X}}+\|\phi(t,\cdot)\|_{\mathcal{X}}^{2}, \quad \forall\ t\in\mathbb{R}.
\end{align}
Furthermore, we also hope to prove that $\mathcal{X}$ and $\mathcal{Y}^{\pm}$ are compatible with the future (or past) radiation field map. In other words, 
\begin{align*}
\mathcal{F}^{+}_{Dirac}: \mathcal{X}\rightarrow \mathcal{Y}^{+}\quad\text{and}\quad \mathcal{F}^{-}_{Dirac}: \mathcal{X}\rightarrow \mathcal{Y}^{-},
\end{align*}
are both isomorphisms. In Minkowski spacetime $\mathbb{R}^{1+3}$, it is known that the Dirac operator commutes with the conformal Killing vector fields for spinors \cite{MR3362023,MR4184660}. Let us recall the expressions of the conformal Killing vectors for spinors in $\mathbb{R}^{1+3}$.
\begin{itemize}
\item Translation vector fields: 
\begin{align*}
\big\{\partial_{t},\partial_{1},\partial_{2},\partial_{3}\big\},
\end{align*}
\item Rotation vector fields:  $\Omega_{ij}=x^{i}\partial_{j}-x^{j}\partial_{i}$
\begin{align*}
\Big\{\Omega_{12},\ \ \Omega_{23},\ \ \Omega_{31}\Big\},
\end{align*}
\item Lorentz boots vector fields: $\Omega_{0i}=t\partial_{i}+x^{i}\partial_{t}$.
\begin{align*}
\Big\{\Omega_{01},\ \ \Omega_{02},\ \ \Omega_{03}\Big\},
\end{align*}
\item Scaling vector fields:  $\Big\{S=t\partial_{t}+r\partial_{r}\Big\}$.
\end{itemize}
which have played important roles in the study of nonlinear wave equations \cite{MR865359}.

To construct $\mathcal{X}$, we need the vector fields for spinors that are compatible with the symmetry of Dirac equation from \cite{MR3362023,MR4184660}, whose motivations will be detailed in section three.
\begin{itemize}
\item Translation vector fields: 
\begin{align*}
\mathcal{T}:=\big\{\partial_{t},\partial_{1},\partial_{2},\partial_{3}\big\},
\end{align*}
\item Modified rotation vector fields:  
\begin{align*}
\mathcal{R}:=\Big\{\Omega_{12}-\frac{1}{2}\gamma^{1}\gamma^{2},\ \ \Omega_{23}-\frac{1}{2}\gamma^{2}\gamma^{3},\ \ \Omega_{31}-\frac{1}{2}\gamma^{3}\gamma^{1}\Big\},
\end{align*}
\item Modified Lorentz boots vector fields:
\begin{align*}
\mathcal{L}:=\Big\{\Omega_{01}-\frac{1}{2}\gamma^{0}\gamma^{1},\ \ \Omega_{02}-\frac{1}{2}\gamma^{0}\gamma^{2},\ \ \Omega_{03}-\frac{1}{2}\gamma^{0}\gamma^{3}\Big\},
\end{align*}
\item Scaling vector fields:  $\mathcal{S}=\Big\{S=t\partial_{t}+r\partial_{r}\Big\}$.
\end{itemize}
Collect the aforementioned vector fields together and denote them as
\begin{align}\label{dirac: conformal vector fields}
\mathcal{A}:=\mathcal{T}\cup\mathcal{R}\cup\mathcal{L}\cup\mathcal{S},
\end{align}
Then the weighted energy space $\mathcal{X}$ of the initial data can be defined as the completion of $C_{0}^{\infty}(\mathbb{R}^{3},\mathbb{C}^{4})$ under the norm
\begin{align*}
\|\phi_{0}\|_{\mathcal{X}}^{2}:=\sum_{|\alpha|\leq 6, \Gamma\in \mathcal{A}}\big\|\Gamma^{\alpha} \phi(0,\cdot)\big\|^{2}_{L^{2}},
\end{align*}
where $\phi(t,x)$ is the solution of Eq. \eqref{dirac: dirac equation} with vanishing source term ($\Phi\equiv 0$).

The existence of small data global solution has been proved in \cite{MR3362023} and refined by \cite{MR4184660}. For the convenience of readers, we summarize this result below.
\begin{proposition}\label{dirac: global existence for dirac equation}
There exists a constant $\varepsilon_{0}>0$, independent of the initial data, such that if the initial data of Eq. \eqref{dirac: semi-linear dirac equation} satisfies  $\|\phi_{0}\|_{\mathcal{X}}\leq\varepsilon_{0}$, then Eq. \eqref{dirac: semi-linear dirac equation} has a global solution $\phi(t,x)\in L^{\infty}(\mathbb{R}, \mathcal{X})$. Moreover, we have
\begin{align}
\sup_{t\in \mathbb{R}}\|\phi(t,\cdot)\|_{\mathcal{X}}\lesssim \|\phi_{0}\|_{\mathcal{X}}.
\end{align}
\end{proposition}
\begin{remark}
Since the notations in this paper are different from those in \cite{MR3362023,MR4184660}, we will give a proof of this result for completeness.
\end{remark}

\medskip

{\it The weighted energy space of the radiation field $\mathcal{Y}^{+}$}: To construct a weighted energy space for the future radiation field, denoted as $\mathcal{Y}^+$, we will utilize the vector fields for spinors that are compatible with the symmetry of Dirac spinors at null infinity which will be detailed in section three. The expressions of these vector fields are presented as follows.
\begin{itemize}
\item Translation vector fields at future null infinity: $\hat{\mathcal{T}}:=\{\partial_{s}\}$,
\item Rotation vector fields at future null infinity:
\begin{align*}
\hat{\mathcal{R}}:=\Big\{\Omega_{12}-\frac{1}{2}\gamma^{1}\gamma^{2}, \ \ \Omega_{23}-\frac{1}{2}\gamma^{2}\gamma^{3},\ \ \Omega_{31}-\frac{1}{2}\gamma^{3}\gamma^{1}\Big\},
\end{align*}
\item Lorentz boots vector fields at future null infinity:
\begin{align*}
\hat{\mathcal{L}}:=\Big\{&-\partial_{\omega^{1}}+\omega^{1}s\partial_{s}+\omega^{1}I+\frac{1}{2}\gamma^{0}\gamma^{1},\ \
-\partial_{\omega^{2}}+\omega^{2}s\partial_{s}+\omega^{2}I+\frac{1}{2}\gamma^{0}\gamma^{2},\\
&-\partial_{\omega^{3}}+\omega^{3}s\partial_{s}+\omega^{3}I+\frac{1}{2}\gamma^{0}\gamma^{3}\Big\},
\end{align*}
\item Scaling vector field at future null infinity: $\hat{\mathcal{S}}=\{s\partial_{s}\}$.
\end{itemize}
\begin{remark}
The future null infinity can be identified as the cylinder $\mathbb{R}\times\mathbb{S}^2$. In this context,  $\partial_{s}$ is the standard derivative with respect to the first component of $(s,\omega)\in\mathbb{R}\times\mathbb{S}^{2}$ and $\partial_{\omega^{i}}$ is the covariant derivate with respect to the second component of $(s,\omega)\in\mathbb{R}\times\mathbb{S}^{2}$.  If we take an extrinsic view to consider $\mathbb{S}^{2}=\{x\in\mathbb{R}^{3}: |x|=1\}$, then 
\[\partial_{\omega^{i}}: \psi(s,\omega)\mapsto (\partial_{i}-\omega^{i}\partial_{r})\Big(\psi(s,\frac{x}{|x|})\Big),\]
 with $\omega^{i}=\frac{x^{i}}{|x|}$ and $i=1,2,3$. 
Additionally, $\hat{\Omega}_{ij}=\omega^{i}\partial_{\omega^{j}}-\omega^{j}\partial_{\omega^{i}}$.  
\end{remark}

Collect these vector fields at future null infinity and denote them as 
\begin{align}\label{dirac: conformal vector fields at null infinity}
\hat{\mathcal{A}}:=\hat{\mathcal{T}}\cup\hat{\mathcal{R}}\cup\hat{\mathcal{L}}\cup\hat{\mathcal{S}}, 
\end{align}
\begin{remark}
Vector fields in $\hat{\mathcal{A}}$ are indeed well defined vector fields for future radiation fields of the Dirac spinor. See remark \ref{dirac: why we use these vector field}.
\end{remark}

Then the weighted energy space $\mathcal{Y}^{+}$ can be defined as the completion of $E^{+}$ with respect to the norm
\begin{align*}
\|\psi\|_{\mathcal{Y}^+}^{2}:=\sum_{|\alpha|\leq 6, \hat{\Gamma}\in \hat{\mathcal{A}}}\|\hat{\Gamma}^{\alpha}\psi\|^{2}_{L^{2}}.
\end{align*}

To address the semi-linear problem, we should first establish the linear scattering theory for the free Dirac equation $\mathcal{D}\phi=0$:

\begin{theorem}[Linear Isomorphism Theorem]
For the free Dirac equation $\mathcal{D}\phi=0$,  i.e. Eq. \eqref{dirac: dirac equation} with vanishing source term, the future radiation field map
\begin{align*}
\mathcal{F}_{Dirac}^{+}(\cdot,0):&\ \mathcal{X}\rightarrow \mathcal{Y}^{+} \\
                                                           &\phi_{0}\mapsto \mathcal{F}^{+}(\phi_{0},0),
\end{align*}
is an isomorphism.
\end{theorem}

Based on the linear isomorphism, we can prove a nonlinear local isomorphism since the weighted energy space $\mathcal{X}$ and 
$\mathcal{Y}^{+}$ are also compatible with the nonlinear estimate when we treat the nonlinear term as a perturbation. By employing Duhamel principle and inverse function theorem in Banach space, we can prove that for the semi-linear Dirac equation \eqref{dirac: semi-linear dirac equation} with initial data $\phi_{0}$ and nonlinear term $N(\phi,\phi)$, there exists $\varepsilon_{1}>0$ such that the nonlinear scattering operator
\begin{align*}
\mathcal{F}^{+}_{nonlinear}: B_{\mathcal{X}}(0,\varepsilon_{1})\subseteq &\mathcal{X}\rightarrow U\subseteq\mathcal{Y}^{+},\\
                                                                       &\phi_{0}\mapsto \mathcal{F}^{+}_{Dirac}\big(\phi_{0}, N(\phi,\phi)\big),
\end{align*}
is a local bijection, where $U$ is a neighborhood of $0$ in $\mathcal{Y}^{+}$.

\begin{remark}
Note that in the nonlinear setting of Eq. \eqref{dirac: semi-linear dirac equation}, the weighted energy space $\mathcal{Y}^{-}$ for the past radiation field and the corresponding nonlinear past radiation field map $\mathcal{F}^{-}_{nonlinear}$ can also be constructed by the same method. Therefore,  it can be shown that there exists neighborhood $U^{+}\subset \mathcal{Y}^{+}$ of 0 and neighborhood $U^{-}\subset \mathcal{Y}^{-}$ of 0 such that 
\[\mathcal{S}_{nonlinear}:= (\mathcal{F}^{+}_{nonlinear})\circ(\mathcal{F}^{-}_{nonlinear})^{-1}: U^{-}\rightarrow U^{+},\]
is a bijection.
\end{remark}

\subsection{Notations and main results}
Now let us introduce the necessary notations. In the Minkowski space $\mathbb{R}^{1+3}$, we will use both the standard coordinate $(t,x)=(x^{0},x^{1},x^{2},x^{3})$ and the polar coordinate $(t,r,\omega)$, where $t\in\mathbb{R}$, $r\geq 0$, and $\omega\in\mathbb{S}^{2}$. Here the coordinates are identified by $x^{i}=r\omega^{i}$ for $i=1,2,3$. The outgoing null vector field $L$ and incoming null vector field $\underline{L}$ are defined as 
\begin{align*}
L=\partial_{t}+\partial_{r},\quad
\underline{L}=\partial_{t}-\partial_{r}.
\end{align*}
For any fixed point $(t_{0},x_{0})\in\mathbb{R}^{1+3}$, the corresponding future and past light cones are defined as follows
\begin{align*}
C^{+}_{(t_{0},x_{0})}:=\big\{(t,x): |x-x_{0}|=t-t_{0}, t\geq t_{0}\big\},\\
C^{-}_{(t_{0},x_{0})}:=\big\{(t,x): |x-x_{0}|=t_{0}-t,t\leq t_{0}\big\}.
\end{align*}
The associated solid light cones are given by
\begin{align*}
D^{+}_{(t_{0},x_{0})}:=\big\{(t,x): |x-x_{0}|\leq t-t_{0}, t\geq t_{0}\big\},\\
D^{-}_{(t_{0},x_{0})}:=\big\{(t,x): |x-x_{0}|\leq t_{0}-t,t\leq t_{0}\big\},
\end{align*}
We then have 
\begin{align*}
\partial D^{+}_{(t_{0},x_{0})}=C^{+}_{(t_{0},x_{0})}, \quad \partial D^{-}_{(t_{0},x_{0})}=C^{-}_{(t_{0},x_{0})}.
\end{align*}
Next we consider a fixed point $(t_{0},\omega_{0})\in\mathbb{R}\times\mathbb{S}^{2}$, the associated outgoing and incoming light rays are 
\begin{align*}
l^{+}_{(t_{0},\omega_{0})}:=\big\{(t_{0}+s,s\omega_{0}): s\geq 0\big\},\\
l^{-}_{(t_{0},\omega_{0})}:=\big\{(t_{0}+s,s\omega_{0}): s\leq 0\big\}.
\end{align*}
Then the corresponding future null infinity $\mathcal{I}^{+}$ can be viewed as the collection of outgoing light rays, while the corresponding past null infinity $\mathcal{I}^{-}$ consists of all the incoming light rays:
\begin{align*}
\mathcal{I}^{+}:=\big\{l^{+}_{(t_{0},\omega_{0})}: (t_{0},\omega_{0})\in \mathbb{R}\times\mathbb{S}^{2}\big\},\\
\mathcal{I}^{-}:=\big\{l^{-}_{(t_{0},\omega_{0})}: (t_{0},\omega_{0})\in \mathbb{R}\times\mathbb{S}^{2}\big\},
\end{align*}
We will use the coordinate $(s,\omega)\in \mathbb{R}\times\mathbb{S}^{2}$ to parametrize both $\mathcal{I}^{+}$  and $\mathcal{I}^{-}$.

\medskip

To better understand the spinor bundle on $\mathbb{R}\times\mathbb{S}^{2}$, we introduce the projection operator
\begin{align*}
P(\omega):=\frac{1}{2}(I+\gamma^{0}\gamma^{i}\omega_{i}),\ \ \omega\in\mathbb{S}^{2}
\end{align*}
Here we adopt an extrinsic viewpoint and define $\mathbb{S}^{2}:=\big\{x\in\mathbb{R}^{3}: |x|=1\big\}$ with $\omega^{i}=x^{i}/|x|$. The term $\gamma^0\gamma^i\omega_{i}$ represents a specific linear combination of gamma matrices weighted by the components $\omega^{i}$ of the vector $\omega$. Then the operator $P(\omega)$ essentially projects spinors onto a subspace corresponding to the direction defined by the vector $\omega\in\mathbb{S}^2$. The properties demonstrated below will be crucial in the calculation of spinors.
\begin{lemma}\label{dirac: properties of projection operators}
The projection operator $P(\omega)$ satisfies the following properties:
\begin{itemize}
\item[1)] $I=P(\omega)+P(-\omega)$,
\item[2)] $P(\omega)^{2}=P(\omega),P(-\omega)^{2}=P(-\omega)$,
\item[3)] $P(\omega)P(-\omega)=0$
\end{itemize}
\end{lemma}
\begin{proof}
The first property can be derived through direct calculation,
\begin{align*}
I=\frac{1}{2}(I+\gamma^{0}\gamma^{i}\omega_{i})+\frac{1}{2}(I-\gamma^{0}\gamma^{i}\omega_{i})=P(\omega)+P(-\omega).
\end{align*}
The second property can be derived from direct calculation and the commutation relations of the gamma matrices,
\begin{align*}
&P(\omega)^{2}=\frac{1}{4}(I+2\gamma^{0}\gamma^{i}\omega_{i}+\gamma^{0}\gamma^{i}\omega_{i}\gamma^{0}\gamma^{j}\omega_{j})\\
&=\frac{1}{4}(I+2\gamma^{0}\gamma^{i}\omega_{i}-\gamma^{i}\gamma^{j}\omega_{i}\omega_{j})=\frac{1}{4}\Big(I+2\gamma^{0}\gamma^{i}\omega_{i}-\frac{1}{2}(\gamma^{i}\gamma^{j}+\gamma^{j}\gamma^{i})\omega_{i}\omega_{j}\Big)\\
&=\frac{1}{4}\Big(I+2\gamma^{0}\gamma^{i}\omega_{i}-\frac{1}{2}(-2\eta^{ij}I)\omega_{i}\omega_{j}\Big)=\frac{1}{2}(I+\gamma^{0}\gamma^{i}\omega_{i}).
\end{align*}
The final property can also be derived from direct calculation and the commutation relations of the gamma matrices,
\begin{align*}
&P(\omega)P(-\omega)=\frac{1}{4}(I-\gamma^{0}\gamma^{i}\omega_{i}\gamma^{0}\gamma^{j}\omega_{j})\\
&=\frac{1}{4}(I+\gamma^{i}\gamma^{j}\omega_{i}\omega_{j})=\frac{1}{4}\Big(I+\frac{1}{2}(\gamma^{i}\gamma^{j}+\gamma^{j}\gamma^{i})\omega_{i}\omega_{j}\Big)\\
&=\frac{1}{4}\Big(I+\frac{1}{2}(-2m^{ij}I)\omega_{i}\omega_{j}\Big)=0.
\end{align*}
\end{proof}

We note that the operator $P(\omega)$ induces a decomposition of $C_{0}^{\infty}(\mathbb{R}\times\mathbb{S}^{2},\mathbb{C}^{4})$ as follows:
\begin{align*}
E^{+}:&=\{\psi\in C_{0}^{\infty}(\mathbb{R}\times\mathbb{S}^{2},\mathbb{C}^{4}): P(-\omega)\psi=0 \},\\
E^{-}:&=\{\psi\in C_{0}^{\infty}(\mathbb{R}\times\mathbb{S}^{2},\mathbb{C}^{4}): P(\omega)\psi=0 \},
\end{align*}
The future radiation fields of Dirac equation \eqref{dirac: dirac equation} are elements of $E^{+}$, while the past radiation fields of Dirac equation \eqref{dirac: dirac equation} are elements of $E^{-}$. 

Recall Katayama-Kubo's definition for spinor null form \cite{MR3362023}. A bilinear form $N(\cdot,\cdot):\mathbb{C}^{4}\times\mathbb{C}^{4}\rightarrow \mathbb{C}^{4}$ is called a {\it spinor null form} iff there exists two constant vectors $e_{1},e_{2}\in\mathbb{C}^{4}$ such that:
\begin{align}
N(X,Y)=\left\langle \gamma^{0}X,Y\right\rangle e_{1}+\left\langle \gamma^{0}\gamma^{5}X,Y\right\rangle e_{2},
\end{align}
whose detailed properties are summarized in the following lemma, whose proof are contained in \cite{MR3362023},
\begin{lemma}\label{dirac: spinor null form}
Let $N(\cdot,\cdot)$ be a spinor null form, then:
\begin{itemize}
\item[(i)] $N(P(\omega)X,P(\omega)X)=0$ for any $\omega\in\mathbb{S}^{2}$ and $X\in\mathbb{C}^{4}$.
\item[(ii)] $N(\cdot,\cdot)$ has good commutation property with vector fields in $\mathcal{A}$. For any $\Gamma\in \mathcal{A}$, there exists another spinor null form $N_{\Gamma}(\cdot,\cdot)$ such that:
\begin{align*}
\Gamma(N(\phi_{1},\phi_{1}))=N(\Gamma\phi_{1},\phi_{2})+N(\phi_{1},\Gamma\phi_{2})+N_{\Gamma}(\phi_{1},\phi_{2}),
\end{align*}
where $\phi_{1},\phi_{2}:\mathbb{R}^{1+3}\rightarrow \mathbb{C}^{4}$ are spinors.
\end{itemize}
\end{lemma}

Now let us introduce the energies that will be used in this paper. The classical energy flux through each time slice $\Sigma_{t}$ is defined as 
\begin{align}
 E^{(0)}(t)=\int_{\mathbb{R}^{3}}|\phi(t,x)|^{2}dx,
\end{align}
To define the higher-order energies, we first recall the vector field sets $\mathcal{A}$ and $\hat{\mathcal{A}}$. 
\begin{align*}
   \mathcal{A}=\big\{\partial_{\mu},\ \Omega_{\mu\nu}-\frac{1}{2}\gamma^{\mu}\gamma^{\nu}, \ S\big\}, 
\end{align*}
\begin{align*}
   \hat{\mathcal{A}}=\Big\{\partial_{s},\ \hat{\Omega}_{ij}-\frac{1}{2}\gamma^{i}\gamma^{j},\ -\partial_{\omega^{i}}+\omega^{i}s\partial_{s}+\omega^{i}I+\frac{1}{2}\gamma^{0}\gamma^{i} ,\ s\partial_{s}\Big\},
\end{align*}
where $\mu,\nu=0,1,2,3$ and $i,j=1,2,3$. For any multi-index $\alpha=(\alpha_{1},\alpha_{2},...,\alpha_{11})$, we denote the differential operator $\Gamma_{1}^{\alpha_{1}}...\Gamma_{11}^{\alpha_{11}}$ as $\Gamma_{k}\in\mathcal{A}$ with $k=1,...,11$ and the differential operator $\hat{\Gamma}_{1}^{\alpha_{1}}...\hat{\Gamma}_{11}^{\alpha_{11}}$ as $\hat{\Gamma}_{k}\in\hat{\mathcal{A}}$ with $k=1,...,11$. Then the higher order energy through $\Sigma_{t}$ is 
\begin{align*}
E^{(\alpha)}(t)=\int_{\mathbb{R}^{3}}|\Gamma^{\alpha}\phi(t,x)|^{2}dx.
\end{align*}
For $t\geq 0$, the weighted energy is defined as
\begin{align*}
F^{(0)}(t)=\int_{0}^{t}\int_{\mathbb{R}^{3}}\frac{|P(-\omega)\phi(t,x)|^{2}}{\big(1+|t-|x||\big)^{1+2\mu}}dxdt,
\end{align*}
while the higher order energy flux is given by
\begin{align*}
F^{(\alpha)}(t)=\int_{0}^{t}\int_{\mathbb{R}^{3}}\frac{|P(-\omega)\Gamma^{\alpha}\phi(t,x)|^{2}}{(1+|t-|x||)^{1+2\mu}}dxdt.
\end{align*}
Then for any given $t^{*}\in [0,+\infty]$, the total energy and the total weighted energy, indexed by a integer $k\geq0$, are defined as
\[ E^{k}(t^{*})=\sup_{0\leq t\leq t^{*}}\sum_{|\alpha|=k}E^{(\alpha)}(t)\quad\text{and}\quad
F^{k}(t^{*})=\sup_{0\leq t\leq t^{*}}\sum_{|\alpha|=k}F^{(\alpha)}(t).\]
A crucial aspect of this work is the following a priori energy estimate. These estimates directly lead to the global existence of solutions for the semi-linear Dirac equation with a spinor null form.
\paragraph{Main Energy Estimate} Let $\mu\in (0,\frac{1}{2})$ and $N_{*}\in\mathbb{Z}_{\geq 6}$. There exists a universal constant $\varepsilon_{0}>0$ such that if the initial data $\phi_{0}$ satisfy
\begin{align*}
\mathcal{E}^{N_{*}}(0)=\sum_{|\alpha|\leq N_{*}}\|\Gamma^{\alpha}\phi_{0}\|^{2}_{L^{2}(\mathbb{R}^{3})}\leq \varepsilon_{0}^{2},
\end{align*}
then the semi-linear Dirac equation with spinor null form \eqref{dirac: semi-linear dirac equation} admits a unique global solution $\phi(t,x)$. Furthermore, there is a universal constant $C$ such that the following energy estimates are satisfied:
\[
\sum_{|\alpha|\leq N_{*}}\|\Gamma^{\alpha}\phi(t,\cdot)\|^{2}_{L^{2}(\mathbb{R}^{3})}\leq C\mathcal{E}^{N_{*}}(0)
\]
\[
    \sum_{|\alpha|\leq N_{*}}\int_{0}^{t}\int_{\mathbb{R}^{3}}\frac{|P(-\omega)\Gamma^{\alpha}\phi(t,x)|^{2}}{\big(1+|t-|x||\big)^{1+2\mu}}dxdt\leq C\mathcal{E}^{N_{*}}(0).
\]
\begin{remark}
We note that $\mathcal{E}^{N_{*}}(0)$ is well defined since $\partial_{t}\phi_{0}$ is determined by the equation itself. The proof of the aforementioned main estimate actually yields a more refined result. This demonstrates the null structure of the semi-linear Dirac equation \eqref{dirac: semi-linear dirac equation} with the spinor null form.
\end{remark}

\begin{remark}
It should be noted that these results are valid for all  $N_{*}\in\mathbb{Z}_{\geq 6}$. For the remainder of the paper, we will set 
$N_{*}=6$. We will take $\mu=\frac{1}{4}$.
\end{remark}

\paragraph{Refined Energy Estimates}Let $\mu\in (0,\frac{1}{2})$ and $N_{*}\in\mathbb{Z}_{\geq 6}$. There exists a universal constant $\varepsilon_{0}>0$ such that if the initial data $\phi_{0}$ satisfies
\begin{align*}
\mathcal{E}^{N_{*}}(0)=\sum_{|\alpha|\leq N_{*}}\|\Gamma^{\alpha}\phi_{0}\|^{2}_{L^{2}(\mathbb{R}^{3})}\leq \varepsilon_{0}^{2},
\end{align*}
then the global solution $\phi(t,x)$ to the semi-linear Dirac equation \eqref{dirac: semi-linear dirac equation}  satisfies the following estimates: there is a universal constant $C$ such that 
\begin{align*}
\sum_{|\alpha|\leq N_{*}}\|\Gamma^{\alpha}\phi(t,\cdot)\|^{2}_{L^{2}(\mathbb{R}^{3})}\leq C\mathcal{E}^{N_{*}}(0)+C(\mathcal{E}^{N_{*}}(0))^{2}.
\end{align*}

Assuming that $\phi(t,x)$ is the solution constructed as described above, we will now define the future radiation field associated with
$\phi(t,x)$. 

\paragraph{Existence of Future Radiation Fields.} For the solution $\phi$ obtained in the Main energy estimates, the associated nonlinear term $N(\phi,\phi)$ can be viewed as the source term. Therefore, it is reasonable to define the {\it nonlinear future radiation field} associated to semi-linear Dirac equation \eqref{dirac: semi-linear dirac equation} as follows,
\[\mathcal{F}_{nonlinear}^{+}(\phi_{0}):=\mathcal{F}^{+}_{Dirac}(\phi_{0},N(\phi,\phi)),\]
It will be shown (in Lemma \ref{part1} and Lemma \ref{part2}) that there exists $\varepsilon_{1}>0$ such that \[\mathcal{F}^{+}_{nonlinear}\in C^{1}(B_{\mathcal{X}}(0,\varepsilon_{1}),\mathcal{Y}^{+}),\] in the sense of Fr{\'e}chet derivative. 
Furthermore, it will be shown (in Lemma \ref{dirac: the same definition}) that when $\phi_{0}\in C_{0}^{\infty}(\mathbb{R}^{3},\mathbb{C}^{4})$ and $\|\phi_{0}\|_{\mathcal{X}}\leq \varepsilon_{0}$, $\mathcal{F}_{nonlinear}^{+}(\phi_{0})$ coincides with the following limit
\begin{equation*}
\lim_{r\rightarrow+\infty}r\phi(r+s,r\omega):=r\phi(|s|+s,|s|\omega)+\int_{|s|}^{+\infty}(\partial_{t}+\partial_{r})(r\phi(s+r,r\omega))dr,
\end{equation*}
for almost every $(s,\omega)\in\mathbb{R}\times\mathbb{S}^{2}$.

We are now prepared to present the main results of this paper.
\begin{theorem}[Local Isomorphism Theorem]\label{dirac: main theorem rough version}
 For the semi-linear Dirac equation \eqref{dirac: semi-linear dirac equation} with nonlinear term $N(\phi,\phi)$, there exists $\varepsilon_{2}>0$ such that 
\begin{align*}
\mathcal{F}^{+}_{nonlinear}: B_{\mathcal{X}}(0,\varepsilon_{2})\subseteq &\mathcal{X}\rightarrow U\subseteq\mathcal{Y}^{+},\\
                                                                       &\phi_{0}\mapsto \mathcal{F}^{+}_{Dirac}\big(\phi_{0}, N(\phi,\phi)\big),
\end{align*}
is a $C^{1}$ diffeomorphism, where $U$ is a neighborhood of $0$ in $\mathcal{Y}^{+}$.
\end{theorem}

\begin{remark}
There are three small positive constants $\varepsilon_{0}\geq \varepsilon_{1}\geq \varepsilon_{2}$ that are chosen for different purpose. $\varepsilon_{0}$ is chosen such that Eq. \eqref{dirac: semi-linear dirac equation} admits global solution when $\phi_{0}\in B_{\mathcal{X}}(0,\varepsilon_{0})$. $\varepsilon_{1}$ is chosen such that the nonlinear radiation field map $\mathcal{F}^{+}_{nonlinear}$ is $C^{1}(B_{\mathcal{X}}(0,\varepsilon_{1}),\mathcal{Y}^{+})$. $\varepsilon_{2}$ is chosen by inverse function theorem for the nonlinear radiation field map $\mathcal{F}^{+}_{nonlinear}$ at $0\in\mathcal{X}$.
\end{remark}

The rest of the paper is organized as follows. In section two, we review some technical results from Friedlander \cite{MR583989}. These results will be used to prove the linear isomorphism theorem later. In section three, we examine the symmetry of Dirac spinors and the symmetry of radiation fields at future null infinity. This explains how to find vector fields that are used to construct the weighted energy spaces. The necessary energy estimates include a Ghost weight estimate are presented in section four. We reprove a small-data global solution for Eq. \eqref{dirac: semi-linear dirac equation} in section five. In section six, we study the properties of the radiation field maps in our weighted energy spaces. Based on the above, we prove our main theorem in section seven using a functional framework.

\section{Technical lemmas on radiation fields}
In this section, we will present some lemmas that will be used in proving the linear isomorphism theorem.
The first lemma concerns the asymptotic behavior of wave equation at future null infinity, which is Theorem 1.6 of page 486 in \cite{MR583989}.
\begin{lemma}\label{dirac: smooth profiles}
Let $u(t,x):\mathbb{R}^{1+3}\rightarrow \mathbb{R}$ be a solution of Eq \eqref{dirac: wave equation} with vanishing source term $G\equiv 0$. Then there exits $v\in C^{\infty}([0,\infty)\times\mathbb{R}\times\mathbb{S}^{2})$ such that:
\begin{align*}
v(\rho,s,\omega)=\frac{1}{\rho}u(\frac{1}{\rho}\omega,\frac{1}{\rho}+\tau).
\end{align*}
By strong Huygens principle, $v(\rho,s,\omega)$ has compact support in the second variable.
\end{lemma}

The second lemma concerns the existence of receding waves, which is Lemma 3.5 of page 501 in \cite{MR583989}.
\begin{lemma}\label{dirac: existence of receding waves}
Let $\psi(s,\omega)\in C^{\infty}(\mathbb{R}\times\mathbb{S}^{2})$ and there exists $s_{1}$ such that $\psi\equiv 0$ for $s>s_{1}$. Then there is a unique $u(t,x)\in C^{\infty}(\mathbb{R}^{1+3})$ satisfying:
\begin{itemize}
\item[1)] $u(t,x):\mathbb{R}^{1+3}\rightarrow \mathbb{R}$ be a solution of free wave equation, i,e, $\Box u=0$,
\item[2)] $u(t,x)=0, t-r\geq s_{1}$,
\item[3)] $\lim_{r\rightarrow+\infty}ru(r+s,r\omega)=\psi(s,\omega)$,
\item[4)] $v(\rho,s,\omega)=\frac{1}{\rho}u(\frac{1}{\rho}\omega,\frac{1}{\rho}+\tau)\in C^{\infty}([0,\infty)\times\mathbb{R}\times\mathbb{S}^{2})$.
\end{itemize}
We call such $u(t,x)$ a receding wave.
\end{lemma}

The third lemma concerns the finite energy property of receding waves, which is Lemma 3.12 of page 502 in \cite{MR583989}.
\begin{lemma}\label{dirac: finite energy of receding waves}
Assuming $\eta(s,\omega)\in C^{\infty}_{0}(\mathbb{R}\times\mathbb{S}^{2})$, then
\begin{align*}
\zeta(s,\omega):=-\int_{s}^{+\infty}\eta(\tau,\omega)d\tau,
\end{align*}
satisfies $\zeta(s,\omega)\in C^{\infty}(\mathbb{R}\times\mathbb{S}^{2})$. Moreover there exists $s_{1}$ such that $\psi\equiv 0$ for $s>s_{1}$ and $\partial_{s}\zeta(s,\omega)=\eta(s,\omega)$. Let $u(t,x)\in C^{\infty}(\mathbb{R}^{1+3})$ denote the unique receding wave constructed for $\zeta(s,\omega)$ in previous Lemma \ref{dirac: existence of receding waves}, then the receding wave $u(t,x)$ satisfies: $(u,\partial_{t}u)(0,\cdot)$ is bounded in $\dot{H}^{1}(\mathbb{R}^{3})\times L^{2}(\mathbb{R}^{3})$.
\end{lemma}

We also need Dirac version of Lemma \ref{dirac: smooth profiles}.
\begin{lemma}\label{dirac: smooth profiles for dirac}
Let $\phi(t,x)$ be solution of the Eq \eqref{dirac: dirac equation} with vanishing source term $\Phi\equiv 0$. Then there exits $\xi\in C^{\infty}([0,\infty)\times\mathbb{R}\times\mathbb{S}^{2},\mathbb{C}^{4})$ such that:
\begin{align*}
\xi(\rho,s,\omega)=\frac{1}{\rho}u(\frac{1}{\rho}\omega,\frac{1}{\rho}+\tau).
\end{align*}
By strong Hygens principle, $\xi(\rho,s,\omega)$ has compact support in the second variable.
\end{lemma}
\begin{proof}
Note that $\Box=-\mathcal{D}^{2}$ and $\mathcal{D}\phi=0$, we have $\Box \phi=0$. Apply Lemma \ref{dirac: smooth profiles} to components of $\phi(t,x)$, we can obtain $\xi\in C^{\infty}([0,\infty)\times\mathbb{R}\times\mathbb{S}^{2},\mathbb{C}^{4})$ such that:
\begin{align*}
\xi(\rho,s,\omega)=\frac{1}{\rho}u(\frac{1}{\rho}\omega,\frac{1}{\rho}+\tau).
\end{align*}
\end{proof}

We also need Dirac version of Lemma \ref{dirac: existence of receding waves}.
\begin{lemma}\label{dirac: existence of receding dirac waves}
Assuming $\psi(s,\omega)\in C^{\infty}(\mathbb{R}\times\mathbb{S}^{2},\mathbb{C}^{4})$ and $P(-\omega)\psi(s,\omega)=0$. If there exists $s_{1}$ such that $\psi\equiv 0$ for $s>s_{1}$. Then there is a unique $\phi(t,x)\in C^{\infty}(\mathbb{R}^{1+3},\mathbb{C}^{4})$ satisfying:
\begin{itemize}
\item[1)] $\phi(t,x):\mathbb{R}^{1+3}\rightarrow \mathbb{C}^{4}$ be a solution of free Dirac equation, i,e, $\mathcal{D} u=0$,
\item[2)] $\phi(t,x)=0, t-r\geq s_{1}$,
\item[3)] $\lim_{r\rightarrow+\infty}r\phi(r+s,r\omega)=\psi(s,\omega)$,
\item[4)] $\xi(\rho,s,\omega)=\frac{1}{\rho}\phi(\frac{1}{\rho}\omega,\frac{1}{\rho}+\tau)\in C^{\infty}([0,\infty)\times\mathbb{R}\times\mathbb{S}^{2},\mathbb{C}^{4})$.
\end{itemize}
We call such $\phi(t,x)$ a receding spinor wave.
\end{lemma}
\begin{proof}
To show the existence of receding Dirac wave, we apply Lemma \ref{dirac: existence of receding waves} to each components of $\psi(s,\omega)\in C^{\infty}(\mathbb{R}\times\mathbb{S}^{2},\mathbb{C}^{4})$. Then we can obtain a unique receding wave $u(t,x):\mathbb{R}^{1+3}\rightarrow \mathbb{C}^{4}$. Note that $\Box=-\mathcal{D}^{2}$. Therefore, $\phi(t,x):=-\mathcal{D}u(t,x)$ gives the desired receding Dirac wave. To show the uniqueness of receding Dirac wave, we assume there exists two receding Dirac waves $\phi_{1}(t,x)$ and $\phi_{2}(t,x)$. Then $\phi_{1}(t,x)-\phi_{2}(t,x)$ satisfies free wave equation $\Box(\phi_{1}-\phi_{2})=0$ and its components satisfy the conditions of receding waves of $\psi(s,\omega)\equiv 0\in C^{\infty}(\mathbb{R}\times\mathbb{S}^{2})$ in Lemma \ref{dirac: existence of receding waves}. By uniqueness property, we can conclude that $\phi_{1}\equiv\phi_{2}$.
\end{proof}

We also need Dirac version of Lemma \ref{dirac: finite energy of receding waves}.
\begin{lemma}\label{dirac: finite energy of receding dirac waves}
Take $\psi(s,\omega)\in C^{\infty}_{0}(\mathbb{R}\times\mathbb{S}^{2},\mathbb{C}^{4})$ and satisfying $P(-\omega)\psi(s,\omega)=0$, let $\phi(t,x)\in C^{\infty}(\mathbb{R}^{1+3},\mathbb{C}^{4})$ denote the unique receding Dirac wave constructed for $\psi(s,\omega)$ in previous Lemma \ref{dirac: existence of receding dirac waves}, then the receding dirac wave $\phi(t,x)$ satisfies:
$\phi(0,\cdot)$ is bounded in $L^{2}(\mathbb{R}^{3})$.
\end{lemma}
\begin{proof}
Recall that in the proof of previous Lemma \ref{dirac: existence of receding dirac waves}, the unique receding dirac receding wave is given by $\phi(t,x):=-\mathcal{D}u(t,x)$, where $u(t,x)$ is the receding wave constructed for components of $\psi(s,\omega)\in C^{\infty}_{0}(\mathbb{R}\times\mathbb{S}^{2},\mathbb{C}^{4})$ in Lemma \ref{dirac: existence of receding waves}. Therefore, $\|\phi(0,\cdot)\|_{L^2}=\|\gamma^{0}\partial_{t}u(0,\cdot)+\gamma^{i}\partial_{i}u(0,\cdot)\|_{L^{2}}\lesssim \|(u,\partial_{t}u)(0,\cdot)\|_{\dot{H}^{1}\times L^{2}}$ is bounded.
\end{proof}

\section{Symmetry of Dirac Spinors}
In this section, we will explore the symmetries of the Dirac spinor $\phi(t,x)$ in $\mathbb{R}^{1+3}$ and extend these symmetries to the corresponding radiation fields at null infinity.

\subsection{Spinor Lie derivatives on the Minkowski space $\mathbb{R}^{1+3}$}
{\bf Lie derivatives along the Lorentz vector field:} Recall the Lorentz transformation of the Minkowski space time
which are defined by the real $4\times4$ matrices
\begin{align*}
(M^{\rho\sigma})^{\mu}{}_{\nu}:=\eta^{\rho\mu}\delta^{\sigma}_{\nu}-\eta^{\sigma\mu}\delta^{\rho}_{\nu},
\end{align*}
with Lie bracket
\begin{align*}
[M^{\mu\nu},M^{\rho\sigma}]=M^{\mu\rho}\eta^{\nu\sigma}+M^{\nu\sigma}\eta^{\mu\rho}-M^{\mu\sigma}\eta^{\nu\rho}-M^{\nu\rho}\eta^{\mu\sigma}.
\end{align*}
As we know, $\big\{(M^{\rho\sigma})\big\}$ is a basis of the real Lie algebra $\mathfrak{so}(1,3)$.
\medskip

In order to study the Lorentz transformation of the spinor field on the Minkowski space time $\mathbb{R}^{1+3}$. We consider the complex $4\times4$ matrices 
\begin{align*}
S^{\rho\sigma}:=\frac{1}{4}[\gamma^{\rho},\gamma^{\sigma}]=\frac{1}{2}\gamma^{\rho}\gamma^{\sigma}+\frac{1}{2}\eta^{\rho\sigma}I,
\end{align*}
with Lie bracket
\begin{align*}
[S^{\mu\nu},S^{\rho\sigma}]=S^{\mu\rho}\eta^{\nu\sigma}+S^{\nu\sigma}\eta^{\mu\rho}-S^{\mu\sigma}\eta^{\nu\rho}-S^{\nu\rho}\eta^{\mu\sigma},
\end{align*}
these matrices $\big\{(S^{\rho\sigma})\big\}$ is a basis of real Lie algebra $\mathfrak{sl}(2,\mathbb{C})$.
\medskip

There is a Lie algebra isomorphism from $\mathfrak{so}(1,3)\mapsto\mathfrak{sl}(2,\mathbb{C})$ from the map $S^{\rho\sigma}\mapsto M^{\rho\sigma}$.  We substitute $t$ as $x^0$ and let $x$ be the four vector $(x^0, x^1, x^2, x^3)^{T}\in\mathbb{R}^{1+3}$. 
Then the Lorentz transformation of the spinor field $\phi(x)$ can be expressed as the formula 
\begin{align*}
\phi^{\alpha}(x)\mapsto \tilde{\phi}^{\alpha}(x):=S(\Lambda)^{\alpha}{}_{\beta}\phi^{\beta}\big(\Lambda^{-1}(x)\big),
\end{align*}
where $\Lambda$ and $S(\Lambda)$ are 
\begin{align*}
\Lambda=exp(\frac{1}{2}A_{\rho\sigma}M^{\rho\sigma}),\ \ S(\Lambda)=exp(\frac{1}{2}A_{\rho\sigma}S^{\rho\sigma}),
\end{align*}
with anti-symmetric coefficients $A_{\rho\sigma}$ satisfying $A_{\rho\sigma}=-A_{\rho\sigma}$. 
Then the Lie derivative of the spinor field $\phi(t,x)$ along the Lorentz vector field $\Omega^{\mu\nu}$ in $\mathbb{R}^{1+3}$ is given by
\begin{align*}
\mathcal{L}_{M^{\rho\sigma}}\phi(x):&=\frac{d}{d\epsilon}\Big(exp(\epsilon S^{\rho\sigma})\phi\big(exp(-\epsilon M^{\rho\sigma})(t, x)\big)\Big)\Big|_{\epsilon=0},\\
&=[(x^{\rho}\partial^{\sigma}-x^{\sigma}\partial^{\rho})I+S^{\rho\sigma}]\phi(x),\\
&=[(x^{\rho}\partial^{\sigma}-x^{\sigma}\partial^{\rho})I+\frac{1}{2}\gamma^{\rho}\gamma^{\sigma}]\phi(x).
\end{align*}
In particular, the Lie derivatives along the Lorentz rotation $\Omega^{ij}$ and Lorentz boost vector field $\Omega^{0i}$ are 
\begin{align*}
\mathcal{L}_{M^{ij}}=(x^{j}\partial_{i}-x^{i}\partial_{j})+\frac{1}{2}\gamma^{i}\gamma^{j}=-\Omega^{ij}+\frac{1}{2}\gamma^{i}\gamma^{j},
\end{align*}
\begin{align*}
\mathcal{L}_{M^{0i}}=-(x^{i}\partial_{t}+t\partial_{i})+\frac{1}{2}\gamma^{0}\gamma^{i}=-\Omega^{0i}+\frac{1}{2}\gamma^{0}\gamma^{i}.
\end{align*}
{\bf Lie derivatives along the translation vector field:} $\mathbb{R}^{1+3}$ gives a natural action on Dirac spinors by translation as follows,
\begin{align*}
\phi(x)\mapsto \tilde{\phi}(x):=\phi(x+x_{0}),
\end{align*}
where $x_{0}\in\mathbb{R}^{1+3}$.
The action is compatible with Dirac operator $\mathcal{D}:=\gamma^{\mu}\partial_{\mu}$ which means
\begin{align*}
\mathcal{D}\phi=\Phi\implies \mathcal{D}\tilde{\phi}=\tilde{\Phi}.
\end{align*}
Therefore, we can calculate the associated spinor Lie algebra as follows,
\begin{align*}
\mathcal{L}_{\frac{\partial}{\partial x^\mu}}\phi(x):=\frac{d}{d\epsilon}\phi\big(x+\epsilon e_{\mu}\big)|_{\epsilon=0}=\partial_{\mu}\phi(x).
\end{align*}
This means that 
\begin{align*}
\mathcal{L}_{\frac{\partial}{\partial x^0}}=\partial_{t}, \quad \mathcal{L}_{\frac{\partial}{\partial x^i}}=\partial_{i}, \  i\in\{1,2,3\}.
\end{align*}
{\bf Lie derivatives along the scaling vector field:}\ 
$\mathbb{R}$ gives a natural action on Dirac spinors by scaling as follows,
\begin{align*}
\phi(x)\mapsto \tilde{\phi}(x):=\phi(exp(-\lambda)x),
\end{align*}
where $\lambda\in\mathbb{R}$.
The action is compatible with Dirac operator $\mathcal{D}:=\gamma^{\mu}\partial_{\mu}$ which means
\begin{align*}
\mathcal{D}\phi=\Phi\implies \mathcal{D}\tilde{\phi}=\tilde{\Phi},
\end{align*}
Therefore, we can calculate the associated spinor Lie algebra as follows,
\begin{align*}
\mathcal{L}_{\frac{\partial}{\partial \lambda}}\phi(x):=\frac{d}{d\epsilon}\phi\big(exp(-\epsilon)x\big)|_{\epsilon=0}=-(x^{\mu}\partial_{\mu})\phi(x).
\end{align*}
Note that 
\begin{align*}
\mathcal{L}_{\frac{\partial}{\partial \lambda}}=-S=-(t\partial_{t}+r\partial_{r}).
\end{align*}

\subsection{Spinor Lie derivatives at null infinity}
In this subsection, we want to extend the symmetry consideration to null infinity.
Let $\phi(t,x)$ be the solution of Eq. \eqref{dirac: dirac equation} with vanishing source term $\Phi\equiv 0$. Recall that the forward radiation field of $\phi(t,x)$ is defined to be the following limit along all the outgoing light rays through the central axis $\{x=0\}$,
\begin{align*}
\mathcal{F}_{Dirac}^{+}(\phi_{0},0)(s,\omega)=\lim_{r\rightarrow+\infty}r\phi(s+r,r\omega).
\end{align*}

If $l^{+}_{t_{0}}$ is an outgoing light ray through the central axis $\{x=0\}$, $\Lambda$ is a time orientation preserving Lorentz transformation, then $\Lambda l^{+}_{t_{0}}$ is still an outgoing light ray but might not through the central axis $\{x=0\}$. Therefore, we need to know the limit of $|x|\phi(t,x)$ along any outgoing light rays. 
\begin{lemma}\label{dirac: light rays}
Let $\phi(t,x)$ be the solution of Eq. \eqref{dirac: dirac equation} with vanishing source term $\Phi\equiv 0$. $\{(t_{0}+\tau,x_{0}+\tau \omega_{0}), \tau\geq 0\}$ is a general outgoing light ray starting from $(t_{0},x_{0})$ with direction $\omega_{0}$. We have,
\begin{align*}
\lim_{\tau\rightarrow+\infty}|x_{0}+\tau \omega_{0}|\phi(t_{0}+\tau,x_{0}+\tau \omega_{0})=\mathcal{F}_{Dirac}^{+}(\phi_{0},0)(t_{0}-\left\langle x_{0},\omega_{0} \right\rangle,\omega_{0}).
\end{align*}
\end{lemma}
\begin{proof}
In view of Lemma \ref{dirac: smooth profiles for dirac}, there exits $\xi\in C^{\infty}([0,\infty)\times\mathbb{R}\times\mathbb{S}^{2},\mathbb{C}^{4})$ such that:
\begin{align*}
\xi(\rho,s,\omega)=\frac{1}{\rho}u(\frac{1}{\rho}\omega,\frac{1}{\rho}+\tau).
\end{align*}
We can write the limit explicitly in terms of $\xi$,
\begin{align*}
&\lim_{\tau\rightarrow+\infty}|x_{0}+\tau \omega_{0}|\phi(t_{0}+\tau,x_{0}+\tau \omega_{0})\\
&=\lim_{\tau\rightarrow+\infty}\xi\Big(\frac{1}{|x_{0}+\tau\omega_{0}|},t_{0}+\tau-|x_{0}+\tau\omega_{0}|,\frac{x_{0}+\tau\omega_{0}}{|x_{0}+\tau\omega_{0}|}\Big),\\
&=\xi\big(0,t_{0}-\left\langle x_{0},\omega_{0} \right\rangle,\omega_{0}\big)=\mathcal{F}_{Dirac}^{+}(\phi_{0},0)\big(t_{0}-\left\langle x_{0},\omega_{0} \right\rangle,\omega_{0}\big).
\end{align*}
\end{proof}
{\bf Lie derivative along the Lorentz vector field: }
The Lie algebra $\mathfrak{so}(1,3)$ has a natural action on Dirac spinor as follows:
\begin{align*}
\phi^{\alpha}(x)\mapsto \tilde{\phi}^{\alpha}(x):=S(\Lambda)^{\alpha}{}_{\beta}\phi^{\beta}(\Lambda^{-1}x),
\end{align*}
where
\begin{align*}
\Lambda:=exp(\frac{1}{2}A_{\rho\sigma}M^{\rho\sigma}),\ S(\Lambda):=exp(\frac{1}{2}A_{\rho\sigma}S^{\rho\sigma}),
\end{align*}
$M_{\rho\sigma}$ are anti-symmetric coefficients satisfying $A_{\rho\sigma}=-A_{\sigma\rho}$. 
By Lemma \ref{dirac: light rays}, the forward radiation field of $\tilde{\phi}$ is given by
\begin{align*}
&\mathcal{F}_{Dirac}^{+}(\tilde{\phi})(s,\omega)=\lim_{r\rightarrow+\infty}r\tilde{\phi}(s+r,r\omega)=\lim_{r\rightarrow +\infty}S(\Lambda)r\phi\Big(\Lambda^{-1}\begin{pmatrix}
s+r\\ r\omega
\end{pmatrix}
\Big),\\
&=\frac{S(\Lambda)}{(\Lambda^{-1})^{0}{}_{\alpha}\omega^{\alpha}}\mathcal{F}_{Dirac}^{+}(\psi)\Big(\frac{s}{(\Lambda^{-1})^{0}{}_{\alpha}\omega^{\alpha}},\frac{(\Lambda^{-1})^{1}{}_{\alpha}\omega^{\alpha}}{(\Lambda^{-1})^{0}{}_{\alpha}\omega^{\alpha}},\frac{(\Lambda^{-1})^{2}{}_{\alpha}\omega^{\alpha}}{(\Lambda^{-1})^{0}{}_{\alpha}\omega^{\alpha}},\frac{(\Lambda^{-1})^{3}{}_{\alpha}\omega^{\alpha}}{(\Lambda^{-1})^{0}{}_{\alpha}\omega^{\alpha}}\Big),
\end{align*}
where $\omega^{0}=1$. To calculate the associated spinor Lie algebra at null infinity, we divide the problems into two cases.
If $(\rho,\sigma)=(i,j)$ with $i,j\in\{1,2,3\}$, we have
\begin{align*}
\mathcal{L}_{M^{ij}}\mathcal{F}_{Dirac}^{+}(\phi)(s,\omega)&=\frac{d}{d\epsilon}\Big|_{\epsilon=0}\Big(exp(\epsilon S^{ij})\mathcal{F}_{Dirac}^{+}(\phi)(s,exp(-\epsilon M^{ij})\omega)\Big),\\
&=\Big((\omega^{j}\partial_{\omega_{i}}-\omega^{i}\partial_{\omega_{j}})I+S^{\rho\sigma}\Big)\mathcal{F}_{Dirac}^{+}(\phi)(s,\omega),\\
&=\Big((\omega^{j}\partial_{\omega_{i}}-\omega^{i}\partial_{\omega_{j}})I+\frac{1}{2}\gamma^{i}\gamma^{j}\Big)\mathcal{F}_{Dirac}^{+}(\phi)(s,\omega).
\end{align*}
If $(\rho,\sigma)=(0,i)$ with $i\in\{1,2,3\}$, since $\omega^{k}\partial_{\omega^{k}}=0$ on the sphere, we have
\begin{align*}
\mathcal{L}_{M^{0i}}\mathcal{F}^{+}_{Dirac}(\phi)(x)&=\frac{d}{d\epsilon}\Big|_{\epsilon=0}\Big(\frac{exp(\epsilon S^{0i})}{(exp(-\epsilon S^{0i}))^{0}{}_{\alpha}\omega^{\alpha}}\mathcal{F}^{+}_{Dirac}(\phi)\big(\frac{s}{(exp(-\epsilon S^{0i}))^{0}{}_{\alpha}\omega^{\alpha}}, \\
&\frac{(exp(-\epsilon S^{0i}))^{1}{}_{\alpha}\omega^{\alpha}}{(exp(-\epsilon S^{0i}))^{0}{}_{\alpha}\omega^{\alpha}},\frac{(exp(-\epsilon S^{0i}))^{2}{}_{\alpha}\omega^{\alpha}}{(exp(-\epsilon S^{0i}))^{0}{}_{\alpha}\omega^{\alpha}},
\frac{(exp(-\epsilon S^{0i}))^{3}{}_{\alpha}\omega^{\alpha}}{(exp(-\epsilon S^{0i}))^{0}{}_{\alpha}\omega^{\alpha}}\big)\Big),\\
&=\Big(S^{0i}+(M^{0i})^{0}{}_{\alpha}\omega^{\alpha}(I+s\partial_{s}+\omega^{k}\partial_{\omega^{k}}),
\\
&-(M^{0i})^{k}{}_{\alpha}\partial_{\omega^{k}}\Big)\mathcal{F}^{+}_{Dirac}(\phi)(s,\omega),\\
&=\big(-\partial_{\omega^{i}}+\omega^{i}s\partial_{s}+\omega^{i}I+\frac{1}{2}\gamma^{0}\gamma^{i}\big)\mathcal{F}_{Dirac}^{+}(\phi)(s,\omega).
\end{align*}
{\bf Translation invariance: }
$\mathbb{R}^{1+3}$ gives a natural action on Dirac spinors by translation as follows,
\begin{align*}
\phi(t,x)\mapsto \tilde{\phi}(t,x):=\phi(t+t_{0},x+x_{0}),
\end{align*}
where $(t_{0},x_{0})\in\mathbb{R}^{1+3}$. Then the forward radiation field of $\tilde{\phi}$ is given by
\begin{align*}
\mathcal{F}_{Dirac}^{+}(\tilde{\phi})(s,\omega)=\mathcal{F}^{+}_{Dirac}(\phi)(s+t_{0}-\left\langle  x_{0},\omega \right\rangle, \omega).
\end{align*}
We can calculate the associated spinor Lie algebra at null infinity as follows, 
\begin{align*}
\mathcal{L}_{\frac{\partial}{\partial t}}\mathcal{F}_{Dirac}^{+}(\phi)(s,\omega):=\frac{d}{d\epsilon}(\mathcal{F}_{Dirac}^{+}(\phi)(s+\epsilon, \omega))|_{\epsilon=0}=\partial_{s}\mathcal{F}_{Dirac}^{+}(\phi)(s,\omega),
\end{align*}
and 
\begin{align*}
\mathcal{L}_{\frac{\partial}{\partial x^{i}}}\mathcal{F}_{Dirac}^{+}(\phi)(s,\omega):=\frac{d}{d\epsilon}(\mathcal{F}^{+}_{Dirac}(\phi)(s-\epsilon\omega^{i}, \omega))|_{\epsilon=0}=-\omega^{i}\partial_{s}\mathcal{F}_{Dirac}^{+}(\phi)(s,\omega).
\end{align*}
{\bf Scaling invariance:} $\mathbb{R}$ gives a natural action on Dirac spinors by scaling as follows,
\begin{align*}
\phi(x)\mapsto \tilde{\phi}(x):=\phi(exp(-\lambda)x),
\end{align*}
where $\lambda\in\mathbb{R}$. Then the forward radiation field of $\tilde{\phi}$ is given by
\begin{align*}
\mathcal{F}_{Dirac}^{+}(\tilde{\phi})(s,\omega)=exp(\lambda)\mathcal{F}^{+}_{Dirac}(\phi)(exp(-\lambda)s,\omega).
\end{align*}
We can calculate the associated spinor Lie algebra at null infinity as follows, 
\begin{align*}
\mathcal{L}_{\frac{\partial}{\partial \lambda}}\mathcal{F}_{Dirac}^{+}(\phi)(s,\omega):&=\frac{d}{d\epsilon}(exp(\epsilon)\mathcal{F}_{Dirac}^{+}(\phi)(exp(-\epsilon)s,\omega))|_{\epsilon=0}\\
&=(I-s\partial_{s})\mathcal{F}_{Dirac}^{+}(\phi)(s,\omega).
\end{align*}
\begin{remark}\label{dirac: why we use these vector field}
In view of the discussions above, the vector fields in $\hat{\mathcal{A}}$ are indeed well defined vector fields for radiation fields of Dirac spinors at null infinity.
\end{remark}


\section{Spinor energy estimates}
In this section, we will present energy inequalities, including the ghost weight energy estimate and the Klainerman-Sobolev inequality, for the Dirac equation.

\begin{lemma}\label{dirac: current-div}
Let $\phi(t,x)$ be the solution of Eq. \eqref{dirac: dirac equation}. Set
\begin{align*}
J^{\mu}(\phi)=\left\langle \gamma^{0}\gamma^{\mu}\phi,\phi\right\rangle,
\end{align*}
to be the charge current. Then we have
\begin{align*}
	\partial_{\mu}J^{\mu}=\left\langle \gamma^{0}\Phi,\phi \right\rangle+\left\langle\gamma^{0}\phi,\Phi\right\rangle.
\end{align*}
 \end{lemma}
\begin{proof}
\begin{align*}
\partial_{\mu}J^{\mu}&=\left\langle \gamma^{0}\gamma^{\mu}\partial_{\mu}\phi,\phi\right\rangle+\left\langle \gamma^{0}\gamma^{\mu}\phi,\partial_{\mu}\phi\right\rangle,\\
&=\left\langle \gamma^{0}\gamma^{\mu}\partial_{\mu}\phi,\phi\right\rangle+\left\langle \gamma^{0}\phi,\gamma^{\mu}\partial_{\mu}\phi\right\rangle,\\
&=\left\langle \gamma^{0}\Phi,\phi\right\rangle+\left\langle \gamma^{0}\phi,\Phi\right\rangle.
\end{align*}
\end{proof}

\begin{lemma}\label{dirac: Classical energy estimate}
Let $\phi(t,x)$ be the solution of Eq. \eqref{dirac: dirac equation}, then for any given $t_{0}<t_{1}$, we have
	\begin{equation*}
		\int_{\mathbb{R}^{3}}|\phi|^{2}(t_{2},x)dx-\int_{\mathbb{R}^{3}}|\phi|^{2}(t_{1},x)dx=\int_{t_{1}}^{t_{2}}\int_{\mathbb{R}^{3}}\left\langle \gamma^{0}\Phi,\phi \right\rangle+\left\langle\gamma^{0}\phi,\Phi\right\rangle dx dt.
	\end{equation*}
\end{lemma}
\begin{proof}
  Integrate $\partial_{\mu}J^{\mu}$ in Lemma \ref{dirac: current-div} on spacetime region $[t_{1},t_{2}]\times \mathbb{R}^{3}$, we have
  \begin{align*}
  \int_{t_{1}}^{t_{2}}\int_{\mathbb{R}^{3}}\left\langle \gamma^{0}\Phi,\phi \right\rangle+\left\langle\gamma^{0}\phi,\Phi\right\rangle dxdt&= \int_{t_{1}}^{t_{2}}\int_{\mathbb{R}^{3}}\partial_{\mu}J^{\mu}dxdt,\\
  &=\int_{\mathbb{R}^{3}}J^{0}(t_{2},x)dx-\int_{\mathbb{R}^{3}}J^{0}(t_{1},x)dx,\\
  &=\int_{\mathbb{R}^{3}}|\phi|^{2}(t_{2},x)dx-\int_{\mathbb{R}^{3}}|\phi|^{2}(t_{1},x)dx.
  \end{align*}
 \end{proof}

\begin{lemma}\label{dirac: energy and flux}
Let $\phi(t,x)$ be the solution of the Eq. \eqref{dirac: dirac equation}, then for any $t_{0}\leq t_{1}$, considering the future solid cone $D^{+}_{(t_{0},0)}\cap\{t_{0}\leq t\leq t_{1}\}$ with boundary:
\begin{align*}
\partial (D^{+}_{(t_{0},0)}\cap\{t_{0}\leq t\leq t_{1}\})=(C^{+}_{(t_{0},0)}\cap\{t_{0}\leq t\leq t_{1}\})\cup\Sigma_{t_{1}}^{r\leq t_{1}-t_{0}},
\end{align*}
we have
\begin{equation}\label{dirac: conservation2}
\begin{split}
	\int_{|x|\leq t_{1}-t_{0}}|\phi|^{2}(t_{1},x)dx&-2\int_{|x|\leq t_{1}-t_{0}}|P(-\omega)\phi|^{2}(t_{0}+|x|,x)dx,\\
	&=\int_{D^{+}_{(t_{0},0)}\cap\{t_{0}\leq t\leq t_{1}\}}\left\langle \gamma^{0}\Phi,\phi \right\rangle+\left\langle\gamma^{0}\phi,\Phi\right\rangle dxdt.
\end{split}
\end{equation}
For any $t_{0}<t_{1}<t_{2}$, considering the future solid truncated cone $D^{+}_{(t_{0},0)}\cap\{t_{1}\leq t\leq t_{2}\}$ with boundary:
\begin{align*}
\partial (D^{+}_{(t_{0},0)}\cap\{t_{0}\leq t\leq t_{1}\})&=(C^{+}_{(t_{0},0)}\cap\{t_{1}\leq t\leq t_{2}\})\cup\Sigma_{t_{1}}^{|x|\leq t_{1}-t_{0}}\cup\Sigma_{t_{2}}^{|x|\leq t_{2}-t_{0}},
\end{align*}
we have 
\begin{equation}\label{dirac: conservation3}
\begin{split}
&\int_{|x|\leq t_{2}-t_{0}}|\phi|^{2}(t_{2},x)dx-\int_{|x|\leq t_{1}-t_{0}}|\phi|^{2}(t_{1},x)dx,\\
&-2\int_{t_{1}-t_{0}\leq|x|\leq t_{2}-t_{0}}|P(-\omega)\phi|^{2}(t_{0}+|x|,x)dx,\\
&=\int_{D^{+}_{(t_{0},0)}\cap\{t_{1}\leq t\leq t_{2}\}}\left\langle \gamma^{0}\Phi,\phi \right\rangle+\left\langle\gamma^{0}\phi,\Phi\right\rangle dxdt.
\end{split}
\end{equation}
\end{lemma}

\begin{proof}
 Integrate $\partial_{\mu}J^{\mu}$ in Lemma \ref{dirac: current-div} on the spacetime region $D^{+}_{(t_{0},0)}\cap\{t_{0}\leq t\leq t_{1}\}$, we can obtain Eq. \eqref{dirac: conservation2} as follows,
\begin{align*}
&\int_{D^{+}_{(t_{0},0)}\cap\{t_{0}\leq t\leq t_{1}\}}\left\langle \gamma^{0}\Phi,\phi \right\rangle+\left\langle\gamma^{0}\phi,\Phi\right\rangle dxdt,\\
&=\int_{D^{+}_{(t_{0},0)}\cap\{t_{0}\leq t\leq t_{1}\}} \partial_{\mu}J^{\mu}(t,x)dxdt,\\
&=\int_{t_{0}}^{t_{1}}\int_{|x|\leq t-t_{0}}\partial_{i}J^{i}(t,x)+\int_{|x|\leq t_{1}-t_{0}}\int_{t_{0}+|x|}^{t_{1}}\partial_{t}J^{0}(t,x)dxdt,\\
&=\int_{t_{0}}^{t_{1}}\int_{|x|=t-t_{0}}J^{i}(t,x)\omega_{i}d\sigma_{x}dt+\int_{|x|\leq t_{1}-t_{0}}J^{0}(t_{1},x)dx-\int_{|x|\leq t_{1}-t_{0}}J^{0}(t_{0}+|x|,x)dx,\\
&=\int_{|x|\leq t_{1}-t_{0}}J^{0}(t_{1},x)dx-\int_{|x|\leq t_{1}-t_{0}}(J^{0}-J^{i}\omega_{i})(t_{0}+|x|,x)dx,\\
&=\int_{|x|\leq t_{1}-t_{0}}|\phi|^{2}(t_{1},x)dx-\int_{|x|\leq t_{1}-t_{0}}\left\langle (I-\gamma^{0}\gamma^{i}\omega_{i})\phi,\phi \right\rangle(t_{0}+|x|,x)dx,\\
&=\int_{|x|\leq t_{1}-t_{0}}|\phi|^{2}(t_{1},x)dx-2\int_{|x|\leq t_{1}-t_{0}}|P(-\omega)\phi|^{2}(t_{0}+|x|,x)dx,
\end{align*}
in the third equality, we use the divergence theorem in sphere $\{|x|\leq t-t_{0}\}$. In the last equality, we use $I=\frac{1}{2}(I+\gamma^{0}\gamma^{i}\omega_{i})+\frac{1}{2}(I-\gamma^{0}\gamma^{i}\omega_{i})=P(\omega)+P(-\omega)$ and $P(\omega)P(-\omega)=P(-\omega)P(\omega)=0$.
In view of Eq. \eqref{dirac: conservation2} with $t_{0},t_{1}$ and $t_{0},t_{2}$, we can obtain Eq. \eqref{dirac: conservation3}.
\begin{align*}
&\int_{|x|\leq t_{2}-t_{0}}|\phi|^{2}(t_{2},x)dx-\int_{|x|\leq t_{2}-t_{0}}|\phi|^{2}(t_{1},x)dx,\\
&-2\int_{t_{1}-t_{0}\leq|x|\leq t_{2}-t_{0}}|P(-\omega)\phi|^{2}(t_{0}+|x|,x)dx,\\
&=\int_{D^{+}_{(t_{0},0)}\cap\{t_{1}\leq t\leq t_{2}\}}\left\langle \gamma^{0}\Phi,\phi \right\rangle+\left\langle\gamma^{0}\phi,\Phi\right\rangle dxdt.
\end{align*}
\end{proof}

\begin{lemma}\label{dirac: conservation2 incoming}
Let $\phi(t,x)$ be the solution of the Eq \eqref{dirac: dirac equation}, then for any $t_{1}\leq t_{0}$, considering the future solid cone $D^{-}_{(t_{0},0)}\cap\{t_{1}\leq t\leq t_{0}\}$ with boundary:
\begin{align*}
\partial (D^{-}_{(t_{0},0)}\cap\{t_{1}\leq t\leq t_{0}\})=(C^{-}_{(t_{0},0)}\cap\{t_{1}\leq t\leq t_{0}\})\cup\Sigma_{t_{1}}^{r\leq t_{0}-t_{1}},
\end{align*}
we have
\begin{equation*}
\begin{split}
	2\int_{|x|\leq t_{0}-t_{1}}|P(\omega)\phi|^{2}(t_{0}-|x|,x)dx&-\int_{|x|\leq t_{0}-t_{1}}|\phi|^{2}(t_{1},x)dx,\\
	&=\int_{D^{-}_{(t_{0},0)}\cap\{t_{1}\leq t\leq t_{0}\}}\left\langle \gamma^{0}\Phi,\phi \right\rangle+\left\langle\gamma^{0}\phi,\Phi\right\rangle dxdt.
\end{split}
\end{equation*}
\end{lemma}

\begin{proof}
The proof is similar to Lemma \ref{dirac: energy and flux}. Integrate $\partial_{\mu}J^{\mu}$ in Lemma \ref{dirac: current-div} on the spacetime region $D^{-}_{(t_{0},0)}\cap\{t_{1}\leq t\leq t_{0}\}$,
\begin{align*}
&\int_{D^{-}_{(t_{0},0)}\cap\{t_{1}\leq t\leq t_{0}\}}\left\langle \gamma^{0}\Phi,\phi \right\rangle+\left\langle\gamma^{0}\phi,\Phi\right\rangle dxdt,\\
&=\int_{D^{-}_{(t_{0},0)}\cap\{t_{1}\leq t\leq t_{0}\}} \partial_{\mu}J^{\mu}(t,x)dxdt,\\
&=\int_{t_{1}}^{t_{0}}\int_{|x|\leq t_{0}-t}\partial_{i}J^{i}(t,x)dxdt+\int_{|x|\leq t_{0}-t_{1}}\int_{t_{1}}^{t_{0}-|x|}\partial_{t}J^{0}(t,x)dxdt,\\
&=\int_{t_{1}}^{t_{0}}\int_{|x|=t_{0}-t}J^{i}(t,x)\omega_{i}d\sigma_{x}dt+\int_{|x|\leq t_{0}-t_{1}}J^{0}(t_{0}-|x|,x)dx-\int_{|x|\leq t_{0}-t_{1}}J^{0}(t_{1},x)dx, \\
&=\int_{|x|\leq t_{0}-t_{1}}(J^{0}+J^{i}\omega_{i})(t_{0}-|x|,x)dx-\int_{|x|\leq t_{0}-t_{1}}J^{0}(t_{1},x)dx,\\
&=\int_{|x|\leq t_{0}-t_{1}}\left\langle (I+\gamma^{0}\gamma^{i}\omega_{i})\phi,\phi \right\rangle(t_{0}-|x|,x)dx-\int_{|x|\leq t_{0}-t_{1}}J^{0}(t_{1},x)dx, \\
&=2\int_{|x|\leq t_{0}-t_{1}}|P(\omega)\phi|^{2}(t_{0}-|x|,x)dx-\int_{|x|\leq t_{0}-t_{1}}J^{0}(t_{1},x)dx,
\end{align*}
in the third equality, we use the divergence theorem in sphere $\{|x|\leq t_{0}-t\}$. In the last equality, we use $I=\frac{1}{2}(I+\gamma^{0}\gamma^{i}\omega_{i})+\frac{1}{2}(I-\gamma^{0}\gamma^{i}\omega_{i})=P(\omega)+P(-\omega)$ and $P(\omega)P(-\omega)=P(-\omega)P(\omega)=0$.
\end{proof}

\begin{lemma}\label{dirac: Ghost weight}(Ghost weight\cite{MR2666888})
Let $\phi(t,x)$ be solution of the Eq. \eqref{dirac: dirac equation}. For any $t_{1}\leq t_{2}$ and $\mu\in (0,\frac{1}{2})$ we have 
\begin{equation*}
\begin{split}
\Big(\int_{t_{1}}^{t_{2}}\int_{\mathbb{R}^{3}}\frac{|P(-\omega)\phi|^{2}}{(1+|t-r|)^{1+2\mu}}dxdt\Big)&\leq \frac{2}{\mu}[\int_{\mathbb{R}^{3}}|\phi|^{2}(t_{1},x)dx+\int_{\mathbb{R}^{3}}|\phi|^{2}(t_{2},x)dx,\\
&+\int_{t_{1}}^{t_{2}}\int_{\mathbb{R}^{3}}\left\langle \gamma^{0}\Phi,\phi \right\rangle+\left\langle\gamma^{0}\phi,\Phi\right\rangle dx dt],
\end{split}
\end{equation*}
where the implicit constant does not depend on the solution.
\end{lemma}

\begin{proof}
Assume $t_{1}\leq t_{0}\leq t_{2}$. Let us multiply $\frac{1}{(1+|t_{0}|)^{1+2\mu}}$ on both side of Eq. \eqref{dirac: conservation2}. Then integrate it with respect to $t_{0}\in [t_{1},t_{2}]$. We obtain,
\begin{align*}
&\int_{t_{1}}^{t_{2}}\int_{|x|\leq t_{2}-t_{0}}\frac{|\phi|^{2}(t_{2},x)}{(1+|t_{0}|)^{1+2\mu}}dx dt_{0},\\
&-2\int_{t_{1}}^{t_{2}}\int_{|x|\leq t_{2}-t_{0}}\frac{|P(-\omega)\phi|^{2}(t_{0}+|x|,x)}{(1+|t_{0}|)^{1+2\mu}}dx dt_{0},\\
&=\int_{t_{1}}^{t_{2}}\int_{D^{+}_{(t_{0},0)}\cap\{t_{0}\leq t\leq t_{2}\}}\frac{\left\langle \gamma^{0}\Phi,\phi \right\rangle+\left\langle\gamma^{0}\phi,\Phi\right\rangle}{(1+|t_{0}|)^{1+2\mu}} dx dt dt_{0}.
\end{align*}
Neglecting the range of $x$ and $t_{0}$, the first term of the left hand side can be bounded by 
\begin{align*}
\int_{\mathbb{R}}\int_{\mathbb{R}^{3}}\frac{|\phi|^{2}(t_{2},x)}{(1+|t_{0}|)^{1+2\mu}}dx dt_{0}\leq \frac{1}{\mu}\int_{\mathbb{R}^{3}}|\phi|^{2}(t_{2},x)dx.
\end{align*}
Set $t=t_{0}+|x|$ and apply Fubini theorem, the second term of the left hand side equals to
\begin{align*}
-2\int_{t_{1}}^{t_{2}}\int_{|x|\leq t-t_{1}}\frac{|P(-\omega)\phi|^{2}(t,x)}{(1+|t-|x||)^{1+2\mu}}dx dt.
\end{align*}
Neglecting the range of $x$ and $t_{0}$, the right hand side term can be bounded by
\begin{align*}
\int_{\mathbb{R}}\int_{t_{1}}^{t_{2}}\int_{\mathbb{R}^{3}}\frac{\left\langle \gamma^{0}\Phi,\phi \right\rangle+\left\langle\gamma^{0}\phi,\Phi\right\rangle}{(1+|t_{0}|)^{1+2\mu}} dx dt, dt_{0}\leq \frac{1}{\mu}\int_{t_{1}}^{t_{2}}\int_{\mathbb{R}^{3}}\left\langle \gamma^{0}\Phi,\phi \right\rangle+\left\langle\gamma^{0}\phi,\Phi\right\rangle dx dt, 
\end{align*}
Taking $t_{0}\leq t_{1}$, multiplying $\frac{1}{(1+|t_{0}|)^{1+2\mu}}$ on both side of Eq. \eqref{dirac: conservation3}, and integrating with respect to $t_{0}\in (-\infty,t_{2}]$ we have
\begin{align*}
&\int_{-\infty}^{t_{1}}\int_{|x|\leq t_{2}-t_{0}}\frac{|\phi|^{2}(t_{2},x)}{(1+|t_{0}|)^{1+2\mu}}dxdt_{0}-\int_{-\infty}^{t_{1}}\int_{|x|\leq t_{1}-t_{0}}\frac{|\phi|^{2}(t_{1},x)}{(1+|t_{0}|)^{1+2\mu}}dxdt_{0},\\
&-2\int_{-\infty}^{t_{1}}\int_{t_{1}-t_{0}\leq|x|\leq t_{2}-t_{0}}\frac{|P(-\omega)\phi|^{2}(t_{0}+|x|,x)}{(1+|t_{0}|)^{1+2\mu}}dxdt_{0},\\
&=\int_{-\infty}^{t_{1}}\int_{D^{+}_{(t_{0},0)}\cap\{t_{1}\leq t\leq t_{2}\}}\frac{\left\langle \gamma^{0}\Phi,\phi \right\rangle+\left\langle\gamma^{0}\phi,\Phi\right\rangle}{(1+|t_{0}|)^{1+2\mu}} dxdtdt_{0}.
\end{align*}
Neglecting the range of $x$ and $t_{0}$, the first term and the second term of the left hand side can be bounded by 
\begin{align*}
\sum_{j=1}^{2}\int_{\mathbb{R}}\int_{\mathbb{R}^{3}}\frac{|\phi|^{2}(t_{j},x)}{(1+|t_{0}|)^{1+2\mu}}dx dt_{0}\leq \sum_{j=1}^{2}\frac{1}{\mu}\int_{\mathbb{R}^{3}}|\phi|^{2}(t_{j},x)dx.
\end{align*}
By Fubini theorem, the third term of the left hand side is equal to
\begin{align*}
-2\int_{t_{1}}^{t_{2}}\int_{|x|\geq t-t_{1}}\frac{|P(-\omega)\phi|^{2}(t,x)}{(1+|t-|x||)^{1+2\mu}}dx dt.
\end{align*}
Neglecting the range of $x$ and $t_{0}$, the right hand side term can be bounded by
\begin{align*}
\int_{\mathbb{R}}\int_{t_{1}}^{t_{2}}\int_{\mathbb{R}^{3}}\frac{\left\langle \gamma^{0}\Phi,\phi \right\rangle+\left\langle\gamma^{0}\phi,\Phi\right\rangle}{(1+|t_{0}|)^{1+2\mu}} dx dt \leq \frac{1}{\mu}\int_{t_{1}}^{t_{2}}\int_{\mathbb{R}^{3}}\left\langle \gamma^{0}\Phi,\phi \right\rangle+\left\langle\gamma^{0}\phi,\Phi\right\rangle dx dt.
\end{align*}
Therefore, we have
\begin{align*}
\int_{t_{1}}^{t_{2}}\int_{\mathbb{R}^{3}}\frac{|P(-\omega)\phi|^{2}}{(1+|t-r|)^{1+2\mu}}dxdt&\leq \frac{2}{\mu}\Big(\int_{\mathbb{R}^{3}}|\phi|^{2}(t_{1},x)dx+\int_{\mathbb{R}^{3}}|\phi|^{2}(t_{2},x)dx,\\
&+\int_{t_{1}}^{t_{2}}\int_{\mathbb{R}^{3}}\left\langle \gamma^{0}\Phi,\phi \right\rangle+\left\langle\gamma^{0}\phi,\Phi\right\rangle dx dt\Big).
\end{align*}
\end{proof}

We also need Klainerman-Sobolev  inequality for spinors.
\begin{lemma}\label{dirac: KS inequality} 
Suppose that $\phi(t,x): \mathbb{R}^{1+3}\rightarrow \mathbb{C}^{4}$. Then we have the following pointwise decay estimate:
\begin{align}\label{dirac: KS inequality for spinors} 
\big(1+|t+r|\big)\big(1+|t-r|\big)^{\frac{1}{2}}|\phi(t,x)|\lesssim \sum_{|\alpha|\leq 2,\hat{\Gamma}\in\mathcal{A}}\|\Gamma^{\alpha}\phi(t,\cdot)\|_{L^{2}(\mathbb{R}^{3})}, t\geq 0.
\end{align}
where $\mathcal{A}$ is the set of conformal vector fields on $\mathbb{R}^{1+3}$ for spinors, see Eq. \eqref{dirac: conformal vector fields}.
\end{lemma}
\begin{proof}
Note that 
\begin{align*}
|(x^{\mu}\partial^{\nu}-x^{\nu}\partial^{\mu}-\frac{1}{2}\gamma^{\mu}\gamma^{\nu})\phi|\lesssim |(x^{\mu}\partial^{\nu}-x^{\nu}\partial^{\mu})\phi|+|\phi|.
\end{align*}
Applying classical Klainerman-Sobolev inequality \cite{MR865359}, we can obtain Eq. \eqref{dirac: KS inequality for spinors}.
\end{proof}

\section{Global solutions of Dirac equations with spinor null forms}
In this section we will prove the main energy estimate which directly lead to a global solution to the semi-linear Dirac equation with spinor null forms.

We assume that the size of the initial energy is given by $\mathcal{E}^{N*}(0)=\varepsilon^{2}$. The parameter $\varepsilon$ is a small positive number and its size will be determined at the end of this section. We will take $N_{*}=6$.
We will prove the Main Energy Estimates by the standard method of continuity. We assume that there exists a $t^{*}$ such that
\begin{align}\label{dirac: 11}
\sum_{k=0}^{N^{*}}E^{k}(t^{*})+\sum_{k=0}^{N*}F^{k}(t^{*})\leq 20\varepsilon^{2}.
\end{align}
This is a legitimate assumption: \eqref{dirac: 11} holds for initial data, hence it remains correct for at least a short time interval $[0,t^{*}]$. To implement the continuity argument, we will show that there exist universal constants $\varepsilon_{0}$, under the assumption  \eqref{dirac: 11}, we can indeed obtain a better bound:
\begin{align}\label{dirac: 111}
\sum_{k=0}^{N^{*}}E^{k}(t^{*})+\sum_{k=0}^{N*}F^{k}(t^{*})\leq 18\varepsilon^{2},
\end{align}
provided that $\varepsilon<\varepsilon_{0}$.

We remark that the constant $\varepsilon_{0}$ does not depend on the lifespan $[0,t_{*}]$. Then the assumption \eqref{dirac: 11} will never be saturated so that we can continue $t^{*}$ to $+\infty$. The global existence of solutions to Eq. \eqref{dirac: semi-linear dirac equation} also follows. Therefore, it suffices to prove \eqref{dirac: 111} under assumption \eqref{dirac: 11}.

To begin with, we use Klainerman-Sobolev inequality and the structure of the equation to obtain the decay estimates for $\phi$.
\begin{lemma}\label{dirac: decay estiamte}
Let $N_{\infty}=N_{*}-2$. We have decay estimate for $\phi$ like
\begin{align}\label{dirac: 22}
\sum_{|\alpha|\leq N_{\infty}}|\Gamma^{\alpha}\phi(t,x)|\lesssim \frac{\varepsilon}{(1+t+|x|)(1+|t-|x||)^{\frac{1}{2}}}
\end{align}
and improved decay estimate for $P(-\omega)\phi$ like
\begin{align}\label{dirac: 222}
\sum_{|\alpha|\leq N_{\infty}-1}|P(-\omega)\Gamma^{\alpha}\phi(t,x)|\lesssim \frac{\varepsilon}{(1+t+|x|)^{\frac{3}{2}}}.
\end{align}
\end{lemma}
\begin{proof}
Combining Lemma \ref{dirac: KS inequality for spinors} with the assumption \eqref{dirac: 11}, we have
\begin{align*}
\sum_{|\alpha|\leq N_{\infty}}|\Gamma^{\alpha}\phi(t,x)|\lesssim \frac{\varepsilon}{(1+t+|x|)(1+|t-|x||)^{\frac{1}{2}}},
\end{align*}
which confirms Eq. \eqref{dirac: 22}. 

In view of Lemma \ref{dirac: spinor null form}, for $|\alpha|\leq N_{\infty}-1$,
\begin{align*}
\mathcal{D}\Gamma^{\alpha}\phi=\sum_{\alpha_{1}+\alpha_{2}\leq \alpha}N_{\alpha_{1},\alpha_{2}}(\Gamma^{\alpha_{1}}\phi,\Gamma^{\alpha_{2}}\phi),
\end{align*}
where $N_{\alpha_{1},\alpha_{2}}$ are spinor null forms. In view of \eqref{dirac: 22}, 
\begin{align*}
|\mathcal{D}\Gamma^{\alpha}\phi|\lesssim \frac{\varepsilon^{2}}{(1+t+|x|)^{2}(1+|t-|x||)}.
\end{align*}
In view of identity $\partial_{i}=\omega^{i}\partial_{r}+(\partial_{i}-\omega^{i}\partial_{r})$, we have
\begin{align*}
(\partial_{t}-\partial_{r})P(-\omega)\Gamma^{\alpha}\phi=-P(-\omega)\gamma^{0}\gamma^{i}(\partial_{i}-\omega^{i}\partial_{r})\Gamma^{\alpha}\phi+P(-\omega)\mathcal{D}\Gamma^{\alpha}\phi.
\end{align*}
Therefore, 
\begin{align*}
|(\partial_{i}-\omega^{i}\partial_{r})\Gamma^{\alpha}\phi|&\lesssim \frac{1}{(1+|x|)}\big(\sum_{1\leq i<j\leq 3}|\Omega_{ij}\Gamma^{\alpha}\phi|+|\Gamma^{\alpha}\phi|\big),\\
&\lesssim \frac{\varepsilon}{(1+|x|)(1+|x|+t)(1+|t-|x||)^{\frac{1}{2}}}.
\end{align*}
Then we have for any $|\alpha|\leq N_{\infty}-1$,
\begin{align}\label{dirac: 3}
|(\partial_{t}-\partial_{r})P(-\omega)\Gamma^{\alpha}\phi|\lesssim \frac{\varepsilon}{(1+|x|)(1+|x|+t)(1+|t-|x||)^{\frac{1}{2}}},
\end{align}
This estimate will help us get better decay of $P(\omega)\phi$ by integrating along incoming rays as follows,
\begin{itemize}
\item[(i)] when $|x|\leq \frac{1}{2}t$ or $|x|\geq 2t$, \eqref{dirac: 22} already implies \eqref{dirac: 222}.
\item[(ii)] when $\frac{1}{2}t\leq|x|\leq 2t$, \eqref{dirac: 3} implies
\begin{align}\label{dirac: 4}
|(\partial_{t}-\partial_{r})P(-\omega)\Gamma^{\alpha}\phi|\lesssim \frac{\varepsilon}{(1+|x|+t)(1+|t-|x||)^{\frac{1}{2}}}.
\end{align}
The incoming light ray $\{(t+\tau,x-\frac{x}{|x|}\tau),\tau\in\mathbb{R}\}$ through $(t,x)$ intersects $|x|=\frac{1}{2}t$ at $\big(\frac{2}{3}(t+|x|),\frac{1}{3}(t+|x|)\frac{x}{|x|}\big)$ and intersects $|x|=2t$ at $\big(\frac{1}{3}(t+|x|),\frac{2}{3}(t+|x|)\frac{x}{|x|}\big)$. The integral of Eq. \eqref{dirac: 4} along this incoming light ray from $\big(\frac{2}{3}(t+|x|),\frac{1}{3}(t+|x|)\frac{x}{|x|}\big)$ to $\big(\frac{1}{3}(t+|x|),\frac{2}{3}(t+|x|)\frac{x}{|x|}\big)$ is bounded by
\begin{align*}
\frac{\varepsilon}{(1+t+|x|)^{2}}\int_{-\frac{1}{3}(t+|x|)}^{\frac{1}{3}(t+|x|)}\frac{1}{(1+|\tau|)^{\frac{1}{2}}}d\tau\lesssim \frac{\varepsilon}{(1+t+|x|)^{\frac{3}{2}}},
\end{align*}
which implies Eq. \eqref{dirac: 222}. 
\end{itemize}
\end{proof}
In view of Lemma \ref{dirac: Classical energy estimate} and Lemma \ref{dirac: Ghost weight}, we have
\begin{equation}\label{dirac: 6}
\begin{split}
\sum_{k=0}^{N_{*}}E^{k}(t^{*})&=\mathcal{E}^{N_{*}}(0)+\sum_{|\alpha_{1}|+|\alpha_{2}|\leq N_{*}}\int_{0}^{t^{*}}\int_{\mathbb{R}^{3}}N_{\alpha_{1},\alpha_{2}}(\Gamma^{\alpha_{1}}\phi,\Gamma^{\alpha_{2}}\phi)\\&\big(\left\langle\gamma^{0}e_{\alpha_{1},e_{\alpha_{2}}},\Gamma^{\alpha_{1}+\alpha_{2}}\phi\right\rangle+\left\langle\gamma^{0}\Gamma^{\alpha_{1}+\alpha_{2}}\phi,e_{\alpha_{1},\alpha_{2}}\right\rangle\big)dxdt,
\end{split}
\end{equation}
and
\begin{equation}\label{dirac: 7}
\begin{split}
\sum_{k=0}^{N_{*}}F^{k}(t^{*})&=\frac{2}{\mu}\big(\sum_{k=0}^{N_{*}}E^{k}(t^{*}\big)+\mathcal{E}^{N_{*}}(0)\big)+\sum_{|\alpha_{1}|+|\alpha_{2}|\leq N_{*}}\int_{0}^{t^{*}}\int_{\mathbb{R}^{3}}N_{\alpha_{1},\alpha_{2}}(\Gamma^{\alpha_{1}}\phi,\Gamma^{\alpha_{2}}\phi)\\&\big(\left\langle\gamma^{0}e_{\alpha_{1},e_{\alpha_{2}}},\Gamma^{\alpha_{1}+\alpha_{2}}\phi\right\rangle+\left\langle\gamma^{0}\Gamma^{\alpha_{1}+\alpha_{2}}\phi,e_{\alpha_{1},\alpha_{2}}\right\rangle\big)dxdt.
\end{split}
\end{equation}

Note that the bulk terms have similar structure and can be bounded via the following lemma.
\begin{lemma}
Under the assumption \eqref{dirac: 11}, we have
\begin{align}
\sum_{|\alpha_{1}|+|\alpha_{2}|\leq N_{*}}\int_{0}^{t^{*}}\int_{\mathbb{R}^{3}}|N_{\alpha_{1},\alpha_{2}}(\Gamma^{\alpha_{1}}\phi,\Gamma^{\alpha_{2}}\phi)||\Gamma^{\alpha_{1}+\alpha_{2}}\phi|dxdt\lesssim \varepsilon^{3}.
\end{align}
\end{lemma}

\begin{proof}
Since $N_{*}\geq 6$, 
\begin{align*}
N_{\infty}-1=(N_{*}-3)\geq \frac{1}{2}N_{*},
\end{align*}
it suffices to show
\begin{align}
\sum_{|\alpha_{1}|+|\alpha_{2}|\leq N_{*},|\alpha_{1}|\leq N_{\infty}-1}\int_{0}^{t^{*}}\int_{\mathbb{R}^{3}}|N_{\alpha_{1},\alpha_{2}}(\Gamma^{\alpha_{1}}\phi,\Gamma^{\alpha_{2}}\phi)||\Gamma^{\alpha_{1}+\alpha_{2}}\phi|dxdt\lesssim \varepsilon^{3}.
\end{align}
In view of Lemma \ref{dirac: spinor null form}, 
\begin{align*}
N_{\alpha_{1},\alpha_{2}}\big(\Gamma^{\alpha_{1}}\phi,\Gamma^{\alpha_{2}}\phi\big)&=N_{\alpha_{1},\alpha_{2}}\big(P(-\omega)\Gamma^{\alpha_{1}}\phi,P(\omega)\Gamma^{\alpha_{2}}\phi\big),\\
&+N_{\alpha_{1},\alpha_{2}}\big(P(-\omega)\Gamma^{\alpha_{1}}\phi,P(\omega)\Gamma^{\alpha_{2}}\phi\big).
\end{align*}
There are two cases,
\begin{itemize}
\item[(i)] $|\alpha_{1}|+|\alpha_{2}|\leq N_{*},|\alpha_{1}|\leq N_{\infty}-1$, we use improved decay estimate \eqref{dirac: 222} for $P(-\omega)\Gamma^{\alpha_{1}}\phi_{1}$,
\begin{align*}
&\int_{0}^{t^{*}}\int_{\mathbb{R}^{3}}|N_{\alpha_{1},\alpha_{2}}\big(P(-\omega)\Gamma^{\alpha_{1}}\phi,P(\omega)\Gamma^{\alpha_{2}}\phi\big)||\Gamma^{\alpha_{1}+\alpha_{2}}\phi|dxdt,\\
&\lesssim \int_{0}^{t^{*}}\|P(-\omega)\Gamma^{\alpha_{1}}\phi(t,\cdot)\|_{L^{\infty}}E^{k}(t)dt,\\
&\lesssim \int_{0}^{t^{*}}\frac{\varepsilon^{3}}{(1+t)^{\frac{3}{2}}}dt\lesssim \varepsilon^{3}.
\end{align*}
\item[(ii)]  $|\alpha_{1}|+|\alpha_{2}|\leq N_{*},|\alpha_{1}|\leq N_{\infty}-1$. If $|\alpha_{2}|\leq N_{\infty}-1$, it can be handled the same as the case 1. Otherwise, 
\begin{align*}
|\alpha_{2}|\geq N_{\infty}=N_{*}-2\implies |\alpha_{1}|\leq 2\leq N_{\infty},
\end{align*}
we use decay estimate \eqref{dirac: 22} for $P(\omega)\Gamma^{\alpha_{1}}\phi_{1}$ and weighted space time norm for $P(-\omega)\Gamma^{\alpha_{1}}\phi_{2}$,
\begin{align*}
&\int_{0}^{t^{*}}\int_{\mathbb{R}^{3}}|N_{\alpha_{1},\alpha_{2}}\big(P(\omega)\Gamma^{\alpha_{1}}\phi,P(-\omega)\Gamma^{\alpha_{2}}\phi\big)||\Gamma^{\alpha_{1}+\alpha_{2}}\phi|dxdt,\\
&=\int_{0}^{t^{*}}\int_{\mathbb{R}^{3}}[(1+|t-|x||)^{\frac{1}{2}+\mu}|P(\omega)\Gamma^{\alpha_{1}}\phi||\Gamma^{\alpha_{1}+\alpha_{2}}\phi|]\frac{|P(-\omega)\Gamma^{\alpha_{2}}\phi|}{(1+|t-|x||)^{\frac{1}{2}+\mu}}dxdt,\\
&\lesssim \int_{0}^{t^{*}}\int_{\mathbb{R}^{3}}\frac{\varepsilon|\Gamma^{\alpha_{1}+\alpha_{2}}\phi|}{(1+t)^{1-\mu}}\frac{|P(-\omega)\Gamma^{\alpha_{2}}\phi|}{(1+|t-|x||)^{\frac{1}{2}+\mu}}dxdt,\\
&\lesssim\big(\int_{0}^{t^{*}}\frac{\varepsilon^{2}(\sum_{k=0}^{N_{*}}E^{k}(t)}{(1+t)^{2-2\mu}}dt\big)^{\frac{1}{2}}\big(\sum_{k=0}^{N_{*}}F^{k}(t^{*})\big)^{\frac{1}{2}}\lesssim  \varepsilon^{3}.
\end{align*}
\end{itemize}
In view of Eq. \eqref{dirac: 6} and Eq. \eqref{dirac: 7}, we have
\begin{equation*}
\begin{cases}
\sum_{k=0}^{N_{*}}E^{k}(t^{*})=\varepsilon^{2}+O(\varepsilon^{3}),\\
\sum_{k=0}^{N_{*}}F^{k}(t^{*})=\frac{4}{\mu}\varepsilon^{2}+O(\varepsilon^{3}).
\end{cases}
\end{equation*}
Since $\mu=\frac{1}{4}$, we have
\begin{align*}
\sum_{k=0}^{N_{*}}E^{k}(t^{*})+\sum_{k=0}^{N_{*}}F^{k}(t^{*})\leq 17\varepsilon^{2}+O(\varepsilon^{3})\leq 18\varepsilon^{2}, \quad\forall\ \varepsilon\leq \varepsilon_{0},
\end{align*}
provided that $\varepsilon_{0}$ is sufficiently small. 
\end{proof}

\section{Radiation fields of homogenous and inhomogenous Dirac equation}
In this section, we aim to analyze the linear radiation fields of Dirac equations in two cases. For the homogeneous Dirac equation, the radiation field map is an isomorphism. For the inhomogeneous equation, the radiation field map is a bilinear map which will be crucial to deal with the nonlinear radiation fields.
\subsection{Homogenous case}
We first show that for any solution of homogenous Dirac equation with smooth and compactly supported initial data, there exists an associated profile pair.
\begin{lemma}[Existence and uniqueness of the smooth profile pair]\label{dirac: dirac profile}
Let $\phi(t,x)$ be solution of Eq. \eqref{dirac: dirac equation} with vanishing source term $\Phi\equiv0$. Then there exist two unique smooth functions $\xi_{1}(\rho,s,\omega),\xi_{2}(\rho,s,\omega)\in C^{\infty}([0,\infty)\times\mathbb{R}\times\mathbb{S}^{2},\mathbb{C}^{4})$ with compact support in the second variable such that 
\begin{align*}
	\phi(t,x)=\frac{1}{|x|}\xi_{1}(\frac{1}{|x|},t-|x|,\frac{x}{|x|})+\frac{1}{|x|^{2}}\xi_{2}(\frac{1}{|x|},t-|x|,\frac{x}{|x|})\quad t\geq 0,|x|>1
\end{align*}
Furthermore, $\xi_{1}(\rho,s,\omega)$ and $\xi_{2}(\rho,s,\omega)$ satisfy the following identities
\begin{align*}
	P(\omega)\xi_{1}(q,s,\omega)=\xi_{1}(q,s,\omega),\quad P(-\omega)\xi_{2}(q,s,\omega)=\xi_{2}(q,s,\omega).
\end{align*}
\end{lemma}

\begin{proof}
In view of Lemma \ref{dirac: smooth profiles for dirac}, there exits $\xi\in C^{\infty}([0,\infty)\times\mathbb{R}\times\mathbb{S}^{2},\mathbb{C}^{4})$ such that:
\begin{align*}
\xi(\rho,s,\omega)=\frac{1}{\rho}u(\frac{1}{\rho}\omega,\frac{1}{\rho}+\tau).
\end{align*}
Therefore, we have
\begin{align*}
\mathcal{F}^{+}(\phi_{0},0)(s,\omega)=\lim_{r\rightarrow +\infty}r\phi(r+s,r\omega)=\xi(0,s,\omega).
\end{align*}
In view of Lemma \ref{dirac: decay estiamte}, we have 
\begin{align*}
\lim_{r\rightarrow +\infty}rP(-\omega)\phi(r+s,r\omega)=0\implies P(-\omega)\xi(0,s,\omega)=0,
\end{align*}
we can write $P(-\omega)\xi(\rho,s,\omega)$ as follows
\begin{align*}
P(-\omega)\xi(\rho,s,\omega)&=P(-\omega)\xi(0,s,\omega)+\rho\int_{0}^{1}P(-\omega)(\partial_{\rho}\xi)(\rho \tau,s,\omega)d\tau,\\
&=\rho P(-\omega)\int_{0}^{1}(\partial_{\rho}\xi)(\rho \tau,s,\omega)d\tau.
\end{align*}
Therefore, $\xi_{1}$ and $\xi_{2}$ can be given by
\begin{align*}
\xi_{1}(\rho,s,\omega):=P(\omega)\xi(\rho,s,\omega),\quad \xi_{2}(\rho,s,\omega):=P(-\omega)\int_{0}^{1}(\partial_{\rho}\xi)(\rho \tau,s,\omega)d\tau.
\end{align*}
To prove uniqueness, it suffices to note that $P(\omega)\phi$ coincides with $\xi_{1}$ on a dense subset  and $P(-\omega)\phi$ coincides with $\xi_{1}$ on a dense subset.
\end{proof}

\begin{lemma}\label{dirac: 0-order energy}
Let $\phi(t,x)$ be solution of Eq. \eqref{dirac: dirac equation} with vanishing source term $\Phi\equiv0$. Then the radiation field 
\begin{align*}
\mathcal{F}_{Dirac}^{+}(\phi_{0},0)(s,\omega)=\lim_{r\rightarrow +\infty}r\phi(r+s,r\omega),
\end{align*}
exists. Furthermore, we have 
\begin{align*}
\mathcal{F}_{Dirac}^{+}(\phi_{0},0)(s,\omega)=P(\omega)\mathcal{F}^{+}(\phi_{0},0)(s,\omega),
\end{align*}
and
\begin{align*}
\int_{\mathbb{R}}\int_{\mathbb{S}^{2}}|\mathcal{F}^{+}(\phi_{0},0)|^{2}(s,\omega)d\omega ds=\int_{\mathbb{R}^{3}}|\phi_{0}|^{2}(x)dx.
\end{align*}
\end{lemma}

\begin{proof}
In view of Eq. \eqref{dirac: conservation2 incoming} with $\Phi\equiv0$ and $t_{1}=0$, we have
\begin{align}\label{asd}
2\int_{|x|\leq t_{0}}|P(\omega)\phi|^{2}(t_{0}-|x|,x)dx&=\int_{|x|\leq t_{0}}|\phi|^{2}(0,x)dx.
\end{align}
In $(s,\omega)$ coordinate, 
\begin{align*}
x=\frac{t_{0}-s}{2}, s\in [-t_{0},t_{0}],\omega\in\mathbb{S}^{2},
\end{align*}
the Eq. \eqref{asd} is equivalent to
\begin{align*}
\int_{-t_{0}}^{t_{0}}\int_{\mathbb{S}^{2}}|P(\omega)\phi|^{2}(\frac{t_{0}+s}{2},\frac{t_{0}-s}{2}\omega)(\frac{t_{0}-s}{2})^{2}dsd\omega=\int_{|x|\leq t_{0}}|\phi|^{2}(0,x)dx.
\end{align*}
In view of Lemma \ref{dirac: dirac profile}, we can obtain
\begin{align*}
\mathcal{F}^{+}(\phi_{0},0)(s,\omega)=\lim_{r\rightarrow +\infty}r\phi(r+s,r\omega)=\lim_{r\rightarrow +\infty}rP(\omega)\phi(r+s,r\omega).
\end{align*}
By bounded convergence theorem with $t_{0}\rightarrow+\infty$, we can obtain
\begin{align*}
\int_{\mathbb{R}}\int_{\mathbb{S}^{2}}|\mathcal{F}^{+}(\phi_{0},0)|^{2}(s,\omega)d\omega ds=\int_{\mathbb{R}^{3}}|\phi|^{2}(0,x)dx=\int_{\mathbb{R}^{3}}|\phi_{0}|^{2}(x)dx.
\end{align*}
\end{proof}

\begin{lemma}\label{dirac: Dirac-relation}
	Let $\phi(t,x)$ be solution of Eq. \eqref{dirac: dirac equation} with vanishing source term $\Phi\equiv0$, 
	\begin{align*}
		&\mathcal{F}^{+}_{Dirac}\big((\partial_{t}\phi)(0,\cdot),0\big)=\partial_{s}\mathcal{F}^{+}_{Dirac}\big(\phi(0,\cdot),0\big),\\
		&\mathcal{F}^{+}_{Dirac}\big((\partial_{i}\phi)(0,\cdot),0\big)=-\omega^{i}\partial_{s}\mathcal{F}^{+}_{Dirac}\big(\phi(0,\cdot),0\big),\\
		&\mathcal{F}^{+}_{Dirac}\big((\Omega_{ij}\phi)(0,\cdot),0\big)=(\omega^{i}\partial_{\omega^{j}}-\omega^{j}\partial_{\omega^{i}}-\frac{1}{2}\gamma^{i}\gamma^{j})\mathcal{F}^{+}_{Dirac}\big(\phi(0,\cdot),0\big),\\
		&\mathcal{F}^{+}_{Dirac}\big((\Omega_{0i}\phi)(0,\cdot),0\big)=(-\omega^{i}s\partial_{s}-\omega^{i}I+\partial_{\omega^{i}}-\frac{1}{2}\gamma^{0}\gamma^{i})\mathcal{F}^{+}_{Dirac}(\phi,0),\\
		&\mathcal{F}^{+}_{Dirac}\big((S\phi)(0,\cdot),0\big)=(s\partial_{s}-I)\mathcal{F}^{+}_{Dirac}(\phi,0).
	\end{align*}
\end{lemma}

\begin{proof}
In view of Lemma \ref{dirac: dirac profile}, there exist unique two smooth functions $\xi_{1}(\rho,s,\omega),\xi_{2}(\rho,s,\omega)\in C^{\infty}([0,\infty)\times\mathbb{R}\times\mathbb{S}^{2},\mathbb{C}^{4})$ with compact support in the second variable such that 
\begin{align*}
	\phi(t,x)=\frac{1}{|x|}\xi_{1}(\frac{1}{|x|},t-|x|,\frac{x}{|x|})+\frac{1}{|x|^{2}}\xi_{2}(\frac{1}{|x|},t-|x|,\frac{x}{|x|})\quad t\geq 0,|x|>1.
\end{align*}
Since $\mathcal{D}\Gamma \phi=0$ for all $\Gamma\in\mathcal{A}$, we can calculate radiation fields of $\Gamma\phi$ as follows
\begin{itemize}
\item[(i)] $\Gamma=\partial_{t}$
\begin{align*}
\partial_{t}\phi(t,x)=\frac{1}{|x|}\partial_{s}\xi_{1}+O(\frac{1}{|x|^{2}}),
\end{align*}
Therefore, we have
\begin{align*}
\mathcal{F}^{+}_{Dirac}\big((\partial_{t}\phi)(0,\cdot),0)(s,\omega\big)=\partial_{s}\xi_{1}(0,s,\omega).
\end{align*}

\item[(i)] $\Gamma=\partial_{i}$
\begin{align*}
\partial_{i}\phi(t,x)=-\frac{x^{i}}{|x|^{2}}\partial_{s}\xi_{1}+O(\frac{1}{|x|^{2}}).
\end{align*}
Therefore, we have
\begin{align*}
\mathcal{F}^{+}_{Dirac}\big((\partial_{i}\phi)(0,\cdot),0)(s,\omega\big)=-\omega^{i}\partial_{s}\xi_{1}(0,s,\omega).
\end{align*}

\item[(iii)] $\Gamma=\Omega_{ij}$
 \begin{align*}
    	\Omega_{ij}\phi(t,x)=\frac{x^{i}}{|x|^{2}}\partial_{\omega^{j}}\xi_{1}-\frac{x^{j}}{|x|^{2}}\partial_{\omega^{i}}\xi_{1}-\frac{1}{2|x|}\gamma^{i}\gamma^{j}\xi_{1}+O(\frac{1}{|x|^{2}}).
    \end{align*}
Therefore, we have
\begin{align*}
\mathcal{F}^{+}_{Dirac}\big((\partial_{t}\phi)(0,\cdot),0\big)(s,\omega)=(\omega^{i}\partial_{\omega^{j}}-\omega^{j}\partial_{\omega^{i}}-\frac{1}{2}\gamma^{i}\gamma^{j})\xi_{1}(0,s,\omega).
\end{align*}

\item[(iv)] $\Gamma=\Omega_{0i}$
\begin{align*}
\hat{\Omega}_{0i}\phi(t,x)=\frac{1}{|x|}(-\frac{x^{i}}{|x|}s\partial_{s}F_{1}-\frac{x^{i}}{|x|}F_{1}+\partial_{\omega^{i}}F_{1}-\frac{1}{2}\gamma^{0}\gamma^{i}F_{1})+O(\frac{1}{|x|^{2}}).
\end{align*}
Therefore, we have
\begin{align*}
\mathcal{F}^{+}_{Dirac}\big((\Omega_{0i}\phi)(0,\cdot),0\big)(s,\omega)=(-\omega^{i}s\partial_{s}-\omega^{i}I+\partial_{\omega^{i}}-\frac{1}{2}\gamma^{0}\gamma^{i})\xi_{1}(0,s,\omega).
\end{align*}

\item[(v)] $\Gamma=S$
\begin{align*}
S\phi(t,x)=\frac{1}{|x|}(s\partial_{s}\xi_{1}-\xi_{1})+O(\frac{1}{|x|^{2}}).
\end{align*}
Therefore, we have
\begin{align*}
\mathcal{F}^{+}_{Dirac}\big((S\phi)(0,\cdot),0\big)(s,\omega)=(s\partial_{s}-I)\xi_{1}(0,s,\omega).
\end{align*}
\end{itemize}
\end{proof}

\begin{lemma}\label{dirac: embedding}
The radiation field map
\begin{align*}
\mathcal{F}_{Dirac}^{+}:&\ \mathcal{X}\rightarrow \mathcal{Y}^{+}\\
                                       & \phi_{0}\mapsto \mathcal{F}^{+}(\phi_{0},0),
\end{align*}
is a well defined embedding.
\end{lemma}

\begin{proof}
In view of Lemma \ref{dirac: 0-order energy} and Lemma \ref{dirac: Dirac-relation}, for any $\phi_{0}\in C^{\infty}_{0}(\mathbb{R}^{3})$, we have
\begin{align*}
\|\Gamma^{\alpha}\phi(0,\cdot)\|_{L^{2}(\mathbb{S}^{2})}=\|\hat{\Gamma}^{\alpha}\mathcal{F}_{Dirac}^{+}(\phi_{0},0)\|_{L^{2}(\mathbb{S}^{2})},\quad \forall|\alpha|\leq N_{*},
\end{align*}
Therefore, the radiation field map
\begin{align*}
\mathcal{F}_{Dirac}^{+}(\cdot,0): \phi_{0}\mapsto \mathcal{F}_{Dirac}^{+}(\phi_{0},0),
\end{align*}
can be extended to a well defined linear injection from $\mathcal{X}$ to $\mathcal{Y}^{+}$.
\end{proof}

It is well known that even if the radiation field has compact support, the initial data might not have compact support. The following lemma shows that the radiation field map is surjective in suitable weighted energy space.
\begin{lemma}\label{dirac: surjection}
The radiation field map
\begin{align*}
\mathcal{F}_{Dirac}^{+}(\cdot,0): \ &\mathcal{X}\rightarrow \mathcal{Y}^{+}\\
                      &\phi_{0}\mapsto \mathcal{F}_{Dirac}^{+}(\phi_{0},0),
\end{align*}
is also a surjection.
\end{lemma}

\begin{proof}
For any $\psi\in C_{0}^{\infty}(\mathbb{R}\times\mathbb{S}^{2},\mathbb{C}^{4})$ with $P(-\omega)\psi=0$. In view of Lemma \ref{dirac: existence of receding dirac waves} and Lemma \ref{dirac: finite energy of receding dirac waves}, there exists unique receding Dirac wave $\phi(t,x)$ which might not have compact support. However, $\phi$ has finite energy,
\begin{align*}
\|\phi(0,x)\|_{L^{2}(\mathbb{R}^{3})}<\infty.
\end{align*}
For any $\hat{\Gamma}\in \hat{\mathcal{A}}$ and $|\alpha|\leq N_{*}$, $\hat{\Gamma}^{\alpha}\psi\in C_{0}^{\infty}(\mathbb{R}\times\mathbb{S}^{2},\mathbb{C}^{4})$ and $P(-\omega)\hat{\Gamma}^{\alpha}\psi=0$. In view of Lemma \ref{dirac: existence of receding dirac waves} and Lemma \ref{dirac: finite energy of receding dirac waves}, there exists unique receding Dirac wave $\phi^{\alpha}(t,x)$ which might not have compact support. However, $\phi^{\alpha}$ has finite energy,
\begin{align*}
\|\phi^{\alpha}(0,x)\|_{L^{2}(\mathbb{R}^{3})}<\infty.
\end{align*}
On the other hand, $\Gamma^{\alpha}\phi(t,x)$ is also receding Dirac wave for $\hat{\Gamma}^{\alpha}\psi$. By uniqueness of receding wave, we have 
\begin{align*}
\Gamma^{\alpha}\phi(t,x)=\phi^{\alpha}(t,x),
\end{align*}
and 
\begin{align*}
\|\Gamma^{\alpha}\phi(0,x)\|_{L^{2}(\mathbb{R}^{3})}<\infty,\forall |\alpha|\leq N_{*}.
\end{align*}
Therefore, $\forall \delta>0$, there exists $\phi_{1}\in C_{0}^{\infty}(\mathbb{R}^{3},\mathbb{C}^{4})$ such that $\|\phi_{1}-\phi(0,\cdot)\|_{\mathcal{X}}\leq \delta$, then $\|\mathcal{F}_{Dirac}^{+}(\phi_{1},0)-\psi\|_{\mathcal{Y}^{+}}\leq \delta$. Since $\{\psi|\psi\in C_{0}^{\infty}(\mathbb{R}\times\mathbb{S}^{2},\mathbb{C}^{4}),P(-\omega)\psi=0\}$ is dense in $\mathcal{Y}^{+}$, the radiation field map is a surjection.
\end{proof}

In view of Lemma \ref{dirac: embedding} and Lemma \ref{dirac: surjection}, we can obtain
\begin{theorem}\label{dirac: linear isomorphism theorem}
The radiation field map:
\begin{align*}
\mathcal{F}_{Dirac}^{+}(\cdot,0):\ & \mathcal{X}\rightarrow \mathcal{Y}^{+}\\
           & \phi_{0}\mapsto \mathcal{F}^{+}(\phi_{0},0),
\end{align*}
is an isomorphism.
\end{theorem}
 
\subsection{Inhomogenous case}
\begin{lemma}\label{dirac: bilinear map}
Let $\phi(t,x)$ be the solution of Eq. \eqref{dirac: dirac equation}. The future radiation field map for the inhomogenous Dirac equation is defined as
\begin{align*}
\mathcal{F}_{Dirac}^{+}(\phi_{0},\Phi)(s,\omega):=\lim_{r\rightarrow +\infty}r\phi(r+s,r\omega).
\end{align*}
Moreover, this map can be extended to a bilinear map as follows
\begin{align*}
\mathcal{F}_{Dirac}^{+}: \mathcal{X}\times L^{1}(\mathbb{R}^{+},\mathcal{X})\rightarrow \mathcal{Y}^{+},\\
(\phi_{0},\Phi)\mapsto \mathcal{F}^{+}_{Driac}(\phi_{0},\Phi).
\end{align*}
\end{lemma}

\begin{proof}
Let $\phi(t,x)$ be solution of Eq. \eqref{dirac: dirac equation}. Recall that $\phi_{0}\in C_{0}^{\infty}(\mathbb{R}^{3},\mathbb{C}^{4})$ and  $\Phi\in C_{0}^{\infty}(\mathbb{R}^{1+3},\mathbb{C}^{4})$. By Duhamel principle
\begin{align*}
\phi(t,\cdot)=U(t)\phi_{0}+\int_0^{t}U(t-\tau)\Phi(\tau,\cdot)d\tau,
\end{align*}
where $U(t)\phi(0)$ denotes the solution of homogenous Dirac equation with initial data $\phi_{0}$. Since 
\begin{align*}
\mathcal{F}_{Dirac}^{+}\big(U(-\tau)\phi_{0},0\big)(s,\omega)=\mathcal{F}_{Dirac}^{+}(\phi_{0},0)(s-\tau,\omega),  
\end{align*}
we have
\begin{align*}
\mathcal{F}_{Dirac}^{+}(\phi_{0},\Phi)(s,\omega)=\mathcal{F}_{Dirac}^{+}(\phi_{0},0)(s,\omega)+\int_{0}^{+\infty}\mathcal{F}^{+}_{Dirac}\big(\Phi(\tau,\cdot),0\big)(s-\tau,\omega)d\tau.
\end{align*}
For any $\alpha$, in view of Lemma \ref{dirac: Dirac-relation}, we have 
\begin{align*}
&\hat{\Gamma}^{\alpha}\mathcal{F}^{+}_{Dirac}\big(\phi(0,\cdot),\Phi\big)(s,\omega)\\
&=\hat{\Gamma}^{\alpha}\mathcal{F}^{+}_{Dirac}\big(\phi(0,\cdot),0\big)(s,\omega)+\int_{0}^{+\infty}\hat{\Gamma}^{\alpha}\mathcal{F}_{Dirac}^{+}\big(\Phi(\tau,\cdot),0\big)(s-\tau,\omega)d\tau,\\
&=\mathcal{F}_{Dirac}^{+}\big(\Gamma^{\alpha}\phi(0,\cdot),0\big)(s,\omega)+\int_{0}^{+\infty}\mathcal{F}_{Dirac}^{+}\big(\Gamma^{\alpha}\Phi(\tau,\cdot),0\big)(s-\tau,\omega)d\tau.
\end{align*}
Summing over $|\alpha|\leq N_{*}$, we have
\begin{align*}
\big\|\mathcal{F}_{Dirac}^{+}(\phi_{0},\Phi)\big\|_{\mathcal{Y}^{+}}\leq \|\phi_{0}\|_{\mathcal{X}}+\int_{0}^{\infty}\|\Phi(\tau,\cdot)\|_{\mathcal{X}}d\tau.
\end{align*}
Therefore, the radiation field map for inhomogenous Dirac equation can be extended to a bilinear map,
\begin{align*}
\mathcal{F}_{Dirac}^{+}: \mathcal{X}\times L^{1}(\mathbb{R}^{+},\mathcal{X})\rightarrow \mathcal{Y}^{+}, (\phi_{0},\Phi)\mapsto \mathcal{F}^{+}_{Driac}(\phi_{0},\Phi).
\end{align*}
\end{proof}

\section{The proof of Isomorphism theorem}
In this section, we will study the nonlinear radiation fields for Dirac equations with spinor null forms through a functional framework. 

\subsection{Existence of nonlinear radiation field}
We will use $S(\phi_{0}):=\phi$ to denote the solution $\phi$ of Eq. \eqref{dirac: semi-linear dirac equation}
\begin{equation*}
\begin{cases}
\mathcal{D}\phi=N(\phi,\phi),\\
\phi(0,\cdot)=\phi_{0}.
\end{cases}
\end{equation*}
Setting $\mathcal{W}$ to be the completion of $C_{c}^{\infty}(\mathbb{R}^{+}\times\mathbb{R}^{3},\mathbb{C}^{4})$ with respect to the norm
\begin{align*}
\|\phi\|_{\mathcal{W}}^{2}:=\|\phi\|^{2}_{L^{\infty}(\mathbb{R}^{+},\mathcal{X})}+\sum_{|\alpha|\leq N_{*}}\int_{\mathbb{R}^{3}}\frac{\big|P(-\omega)\Gamma^{\alpha}S(\phi_{0})\big|^{2}}{\big(1+|t-|x||\big)^{1+2\mu}}dxdt,
\end{align*}
the nonlinear estimates in section five can be summarized as follows:
\begin{lemma}
There exists a constant $\varepsilon_{0}>0$ such that the initial data to nonlinear solution map
\begin{align*}
B_{\mathcal{X}}(0,\varepsilon_{0})\subset&\mathcal{X}\rightarrow \mathcal{W}\\
                        &\phi_{0}\mapsto S(\phi_{0}),
\end{align*} 
is well defined and $\big\|S(\phi_{0})\big\|_{\mathcal{W}}\lesssim \|\phi_{0}\|_{\mathcal{X}}$.  Furthermore, the nonlinear solution to source term map
\begin{align*}
&\mathcal{W}\rightarrow L^{1}(\mathbb{R}^{+},\mathcal{X})\\
 &\phi \mapsto N(\phi,\phi),
\end{align*}
is well defined and $\big\|N(\phi,\phi)\big\|_{L^{1}(\mathbb{R}^{+},\mathcal{X})}\lesssim \|\phi\|^{2}_{\mathcal{W}}$.

\end{lemma}

The Lemma \ref{dirac: bilinear map} implies that
\begin{lemma}
The initial data and source term to radiation field map,
\begin{align*}
\mathcal{X}\times L^{1}&(\mathbb{R}^{+},\mathcal{X})\rightarrow \mathcal{Y}^{+}\\
                  & (\phi_{0},\Phi)\mapsto \mathcal{F}_{Dirac}^{+}(\phi_{0},\Phi).
\end{align*}
is a well defined bilinear map and
\begin{align*}
\big\|\mathcal{F}^{+}(\phi_{0},\Phi)\big\|_{\mathcal{Y}^{+}}\lesssim \|\phi_{0}\|_{\mathcal{X}}+\|\Phi\|_{L^{1}(\mathbb{R}^{+},\mathcal{X})}.
\end{align*}
\end{lemma}

Note that the nonlinear radiation field map is formed by the composition of three maps: the initial data to nonlinear solution map, the nonlinear solution to source term map, and the initial data and source term to radiation field map. Therefore, we have the lemma below.
\begin{lemma}
There exists a universal $\varepsilon_{0}>0$ such that the nonlinear radiation field map
\begin{align*}
\mathcal{F}^{+}_{nonlinear}: B_{\mathcal{X}}(0,\varepsilon_{0})\subset\mathcal{X}\rightarrow \mathcal{Y}^{+}.
\end{align*}
is well defined and 
\begin{align*}
\big\|\mathcal{F}^{+}_{nonlinear}(\phi_{0})\big\|_{\mathcal{Y}^{+}}&=\Big\|\mathcal{F}^{+}\Big(\phi_{0},N\big(S(\phi_{0}),S(\phi_{0})\big)\Big)\Big\|_{\mathcal{Y}^{+}},\\
&\lesssim \|\phi_{0}\|_{\mathcal{X}}+\Big\|N\big(S(\phi_{0}),S(\phi_{0}))\big)\Big\|_{L^{1}(\mathbb{R}^{+},\mathcal{X})},\\
&\lesssim  \|\phi_{0}\|_{\mathcal{X}}.
\end{align*}
\end{lemma}

\subsection{Proof of the main theorem}
Before the proof of the Main Theorem \ref{dirac: main theorem rough version}, we review the inverse function theorem in Banach space which is classical. The details can be found in \cite{KesavanSrinivasan}. 

Let $E$ and $F$ be real Banach spaces. We denote by $\mathcal{B}(E,F)$ the space of bounded linear transformation of $E$ into $F$. Let $\mathcal{U}\subset E$ be an open set, and $f:\mathcal{U}\rightarrow F$ be a given mapping. The map $f$ is said to be differentiable at $a\in\mathcal{U}$ if there exists a bounded linear transformation $df(a)\in \mathcal{B}(E,F)$ such that
\begin{align*}
\|f(a+h)-f(a)-df(a)h\|_{F}=o(\|h\|_{E}),
\end{align*}
If $df(a)$ exists for each $a\in\mathcal{U}$, we say that $f$ is differentiable in $\mathcal{U}$. If the mapping 
\begin{align*}
\mathcal{U}\rightarrow \mathcal{B}(E,F), a\mapsto df(a),
\end{align*}
is continuous from $\mathcal{U}$ into $\mathcal{B}(E,F)$, we say that $f$ is of class $C^{1}$.
Now we state the inverse function theorem in Banach space (Theorem 1.3.2 in \cite{KesavanSrinivasan}).
\begin{proposition}\label{qwer}
Let $E$ and $F$ be real Banach spaces and $f:\mathcal{U}\subset E\rightarrow F$ be a $C^{1}$-map. Let $a\in\mathcal{U}$ with $f(a)=b$, and let $df(a): E\rightarrow F$ be an isomorphism. Then, there exists a neighborhood $\mathcal{V}_{1}\subset E$ of $a$ and neighborhood $\mathcal{V}_{2}\subset F$ of $b$ such that 
\begin{align*}
f:\mathcal{V}_{1}\rightarrow \mathcal{V}_{2},
\end{align*}
is a $C^{1}$ diffeomorphism. 
\end{proposition}

To prove the Main Theorem \ref{dirac: main theorem rough version}, it suffices to show that there exists constant $\varepsilon_{1}>0$ such that $\mathcal{F}^{+}_{nonlinear}$ is in $C^{1}\big(B_{\mathcal{X}}(0,\varepsilon_{1}\big),\mathcal{Y}^{+})$ and $d\mathcal{F}^{+}_{nonlinear}(0)$ is an isomorphism from $\mathcal{X}$ to $\mathcal{Y}^{+}$.

Firstly, we note that the initial data to nonlinear solution map
\begin{align*}
B_{\mathcal{X}}(0,\varepsilon_{1})\subset\mathcal{X}\rightarrow \mathcal{W}: \phi_{0}\mapsto S(\phi_{0}),
\end{align*} 
is in $C^{1}(B_{\mathcal{X}}(0,\varepsilon_{1}),\mathcal{W})$ provided that $\varepsilon_{1}>0$ is sufficiently small.

\begin{lemma}\label{part1}
If $\varepsilon_{1}$ is small enough,  for any $\phi_{0},\psi_{0}\in B_{\mathcal{X}}(0,\varepsilon_{1})$, the differential map of the initial data to nonlinear solution map at $\phi_{0}$
\begin{align*}
d\mathcal{F}^{+}_{nonlinear}(\phi_{0}): \mathcal{X}\rightarrow \mathcal{W}: \psi_{0}\mapsto \psi,
\end{align*}
is given by the solution of 
\begin{equation}\label{dirac: linearization}
\begin{cases}
\mathcal{D}\psi=N(\phi,\psi)+N(\psi,\phi),\\
\psi(0,\cdot)=\psi_{0},
\end{cases}
\end{equation}
and we have
\begin{align*}
\big\|S(\phi_{0}+\psi_{0})-S(\phi_{0})-d\mathcal{F}^{+}_{nonlinear}(\phi_{0})\psi_{0}\big\|_{\mathcal{W}}\lesssim \|\psi_{0}\|^{2}_{\mathcal{X}}.
\end{align*}
Moreover, the differential map
\begin{align*}
B_{\mathcal{X}}(0,\varepsilon_{1})\rightarrow \mathcal{B}(\mathcal{X},\mathcal{W}), \phi_{0}\mapsto d\mathcal{F}^{+}_{nonlinear}(\phi_{0}),
\end{align*}
is continuous. In particular, the initial data to nonlinear solution map is in $C^{1}\big(B_{\mathcal{X}}(0,\varepsilon_{1}),\mathcal{W}\big)$.
\end{lemma}

\begin{proof}
Writting $\psi:=d\mathcal{F}^{+}_{nonlinear}(\phi_{0})\psi_{0}$, the estimates are based on the following three equations
\begin{itemize}
\item[(i)] Equation for $S(\phi_{0}+\psi_{0})$,
\begin{equation}\label{a}
\begin{cases}
\mathcal{D}S(\phi_{0}+\psi_{0})=N\big(S(\phi_{0}+\psi_{0}),S(\phi_{0}+\psi_{0})\big),\\
S(\phi_{0}+\psi_{0})(0,\cdot)=\phi_{0}+\psi_{0}.
\end{cases}
\end{equation}
\item[(ii)] Equation for $S(\phi_{0})$,
\begin{equation}\label{b}
\begin{cases}
\mathcal{D}S(\phi_{0})=N\big(S(\phi_{0}),S(\phi_{0})\big),\\
S(\phi_{0})(0,\cdot)=\phi_{0}.
\end{cases}
\end{equation}
\item[(iii)] Equation for $\psi$,
\begin{equation}\label{c}
\begin{cases}
\mathcal{D}\psi=N\big(S(\phi_{0}),\psi\big)+N\big(\psi,S(\phi_{0})\big),\\
\psi(0,\cdot)=\psi_{0}.
\end{cases}
\end{equation}
By bootstrap arguments similar to the nonlinear estimates in section five, we can obtain 
\begin{align*}
&\|\psi\|_{\mathcal{W}}\lesssim \|\psi_{0}\|_{\mathcal{X}}+\|S(\phi_{0})\|_{\mathcal{W}}\|\psi\|_{\mathcal{W}}.
\end{align*}
If $\varepsilon_{1}$ is sufficiently small, 
\begin{align*}
\|\psi\|_{\mathcal{W}}\lesssim \|\psi_{0}\|_{\mathcal{X}}.
\end{align*}
\end{itemize}
Substracting Eq. \eqref{b} from Eq. \eqref{a}, we can obtain
\begin{equation}\label{d}
\begin{cases}
&\mathcal{D}\big(S(\phi_{0}+\psi_{0})-S(\phi_{0})\big)=N\big(S(\phi_{0}),S(\phi_{0}+\psi_{0})-S(\phi_{0})\big)\\&+N\big(S(\phi_{0}+\psi_{0})-S(\phi_{0}),S(\phi_{0})\big),\\
&+N\big(S(\phi_{0}+\psi_{0})-S(\phi_{0}),S(\phi_{0}+\psi_{0})-S(\phi_{0})\big),\\
&\big(S(\phi_{0}+\psi_{0})-S(\psi_{0})\big)(0,\cdot)=\psi_{0}.
\end{cases}
\end{equation}
Note that we have arranged the nonlinear terms so that they are either linear in $S(\phi_{0}+\psi_{0})-S(\phi_{0})$ or quadratic in $S(\phi_{0}+\psi_{0})-S(\phi_{0})$. By bootstrap arguments similar to the nonlinear estimates in section five, we can obtain 
\begin{align*}
\|S(\phi_{0}+\psi_{0})-S(\phi_{0})\|_{\mathcal{W}}&\lesssim \|\psi_{0}\|_{\mathcal{W}}+\|S(\phi_{0})\|_{\mathcal{W}}\|S(\phi_{0}+\psi_{0})-S(\phi_{0})\|_{\mathcal{W}}\\
&+\|S(\phi_{0}+\psi_{0})-S(\phi_{0})\|^{2}_{\mathcal{W}}.
\end{align*}
If $\varepsilon_{1}$ is sufficiently small, 
\begin{align*}
\|S(\phi_{0}+\psi_{0})-S(\phi_{0})\|_{\mathcal{W}}\lesssim \|\psi_{0}\|_{X}.
\end{align*}
Substracting Eq. \eqref{c} from Eq. \eqref{d}, we can obtain 
\begin{equation*}
\begin{cases}
&\mathcal{D}\big(S(\phi_{0}+\psi_{0})-S(\phi_{0})-\psi\big)=N\big(S(\phi_{0}),S(\phi_{0}+\psi_{0})-S(\phi_{0})-\psi\big)\\
&+N\big(S(\phi_{0}+\psi_{0})-S(\phi_{0})-\psi,S(\phi_{0})\big)\\
&+N\big(S(\phi_{0}+\psi_{0})-S(\phi_{0}),S(\phi_{0}+\psi_{0})-S(\phi_{0})\big),\\
&\big(S(\phi_{0}+\psi_{0})-S(\phi_{0})-\psi\big)(0,\cdot)=0.
\end{cases}
\end{equation*}
Note that we have arranged the nonlinear terms so that they are either linear in $S(\phi_{0}+\psi_{0})-S(\phi_{0})-\psi$ or quadratic in $S(\phi_{0}+\psi_{0})-S(\phi_{0})$. By bootstrap arguments similar to the nonlinear estimates in section five, we can obtain 
\begin{align*}
\|S(\phi_{0}+\psi_{0})-S(\phi_{0})-\psi\|_{\mathcal{W}}&\lesssim \|S(\phi_{0}+\psi_{0})-S(\phi_{0})-\psi\|_{\mathcal{W}}\|S(\phi_{0})\|_{\mathcal{W}}\\
&+\|S(\phi_{0}+\psi_{0})-S(\phi_{0})\|^{2}_{\mathcal{W}}.
\end{align*}
If $\varepsilon_{1}$ is sufficiently small, 
\begin{align*}
\|S(\phi_{0}+\psi_{0})-S(\phi_{0})-\psi\|_{\mathcal{W}}\lesssim \|S(\phi_{0}+\psi_{0})-S(\phi_{0})\|^{2}_{\mathcal{W}}\lesssim \|\psi_{0}\|_{\mathcal{X}}^{2},
\end{align*}
which shows that the solution map of Eq. \eqref{dirac: linearization} is indeed the differential map of nonlinear solution map $\phi_{0}\mapsto S(\phi_{0})$ at $\phi_{0}\in B_{\mathcal{X}}(0,\varepsilon_{0})$ providing that $\varepsilon_{0}$ is sufficiently small. 

To prove the differential map is continuous, consider equations
\begin{equation}\label{q}
\begin{cases}
\mathcal{D}\psi=N\big(S(\phi_{0}),\psi\big)+N\big(\psi,S(\phi_{0})\big),\\
\psi(0,\cdot)=\psi_{0},
\end{cases}
\end{equation}
and
\begin{equation}\label{w}
\begin{cases}
\mathcal{D}\psi^{'}=N\big(S(\phi^{'}_{0}),\psi^{'}\big)+N\big(\psi^{'},S(\phi^{'}_{0})\big),\\
\psi^{'}(0,\cdot)=\psi^{'}_{0}.
\end{cases}
\end{equation}
Substracting Eq. \eqref{q} from Eq. \eqref{w}, by bootstrap arguments similar to the nonlinear estimates in section five, we can obtain 
\begin{align*}
\|\psi-\psi^{'}\|_{L^{\infty}(\mathbb{R}^{+},\mathcal{X})}\lesssim \|\psi_{0}-\psi^{'}_{0}\|_{\mathcal{X}}+\|\phi_{0}-\phi^{'}_{0}\|_{\mathcal{X}},
\end{align*} 
provided that $\varepsilon_{1}$ is sufficiently small.
\end{proof}

Secondly, the nonlinear solution to source term map is in $C^{1}\big(\mathcal{W},L^{1}(\mathbb{R}^{+},\mathcal{X})\big)$.
\begin{lemma}\label{part2}
The nonlinear solution to source term map:
\begin{align*}
\mathcal{W}\rightarrow L^{1}(\mathbb{R}^{+},\mathcal{X}): \phi \mapsto N(\phi,\phi),
\end{align*}
is in $C^{1}\big(\mathcal{W},L^{1}(\mathbb{R}^{+},\mathcal{X})\big)$. The differential is given by
\begin{align*}
dN(\phi,\phi): &\ \mathcal{W}\rightarrow L^{1}(\mathbb{R}^{+},\mathcal{X}),\\
                          & \psi\mapsto N(\phi,\psi)+N(\psi,\phi).
\end{align*}
\end{lemma}

\begin{proof}
It suffices to note that
\begin{align*}
N(\phi+\psi,\phi+\psi)-N(\phi,\psi)-N(\psi,\phi)-N(\phi,\phi)=N(\psi,\psi).
\end{align*}
\end{proof}

Lastly, the initial data and source term to radiation field map is $C^{1}$ since it is bilinear. By chain rule, the nonlinear radiation field map 
\begin{align*}
\mathcal{F}^{+}_{nonlinear}: \phi_{0}\mapsto \mathcal{F}^{+}\Big(\phi_{0},N\big(S(\phi_{0}),S(\phi_{0})\big)\Big),
\end{align*}
is in $C^{1}\big(B_{\mathcal{X}}(0,\varepsilon_{1}),\mathcal{Y}^{+}\big)$ whose differential map at $\phi_{0}$ is given by
\begin{align*}
d\mathcal{F}^{+}_{nonlinear}(\phi_{0})|\psi_{0}&=\mathcal{F}_{Dirac}^{+}\Big(\psi_{0},N\big(S(\phi_{0}),S(\phi_{0})\big)\Big)\\
&+\mathcal{F}_{Dirac}^{+}\Big(\phi_{0},N\big(S(\phi_{0}),\psi)+N(\psi,S(\phi_{0})\big)\Big),
\end{align*}
where $\psi$ is the solution of equation
\begin{equation*}
\begin{cases}
\mathcal{D}\psi=N\big(S(\phi_{0}),\psi\big)+N\big(\psi,S(\phi_{0})\big),\\
\psi(0,\cdot)=\psi_{0}.
\end{cases}
\end{equation*}
In particular, when $\phi_{0}=0$, $S(\phi_{0})=0$,
\begin{align*}
d\mathcal{F}^{+}_{nonlinear}(0)\psi_{0}=\mathcal{F}_{Dirac}^{+}(\psi_{0},0),
\end{align*}
which is an isomorphism $\mathcal{X}\rightarrow \mathcal{Y}^{+}$ in view of theorem \ref{dirac: linear isomorphism theorem}. By inverse function theorem in Banach space \ref{qwer}, there exists $\varepsilon_{2}>0$ such that 
\begin{align*}
\mathcal{F}^{+}_{nonlinear}: B_{\mathcal{X}}(0,\varepsilon_{2})\subset\mathcal{X}\rightarrow \mathcal{Y}^{+}
\end{align*}
is a local $C^{1}$ diffeomorphism.

\subsection{Justification of nonlinear radiation field}
At last, we show that the nonlinear radiation field map $\mathcal{F}^{+}_{nonlinear}(\phi_{0})$ defined in this paper coincides with 
\begin{align*}
\lim_{r\rightarrow+\infty}r\phi(s+r,r\omega),
\end{align*}
for almost every $(s,\omega)\in \mathbb{R}\times\mathbb{S}^{2}$, where $\phi$ is the solution of Eq. \eqref{dirac: semi-linear dirac equation} with initial data $\phi_{0}$ when $\phi_{0}$ is smooth.

\begin{proposition}\label{dirac: the same definition}
If $\phi_{0}\in C^{\infty}(\mathbb{R}^{3},\mathbb{C}^{4})$ with $\|\phi_{0}\|_{\mathcal{X}}\leq \varepsilon_{0}$, then
\begin{align*}
\lim_{r\rightarrow +\infty}r\phi(r+s,\omega)=\mathcal{F}^{+}_{nonlinear}(\phi_{0})(s,\omega).
\end{align*}
for almost everywhere $(s,\omega)\in\mathbb{R}\times\mathbb{S}^{2}$.
\end{proposition} 
By integrating along the outgoing light rays, we can prove the existence of the limit $\lim_{r\rightarrow +\infty}r\phi(r+s,\omega)$,
\begin{lemma}\label{dirac: 123}
If $\phi(t,x):\mathbb{R}^{1+3}\rightarrow \mathbb{C}^{4}$ is smooth and 
\begin{itemize}
\item[(i)] $\|\phi\|_{L^{\infty}(\mathbb{R}^{+},\mathcal{X})}\lesssim 1$,
\item[(ii)] $|x||\mathcal{D}\phi|\lesssim (1+t+|x|)^{-\frac{3}{2}},t\geq 0$,
\end{itemize}
then we have
\begin{align*}
|(\partial_{t}+\partial_{r})(rP(\omega)\phi)|\lesssim (1+t+|x|)^{-\frac{3}{2}}.
\end{align*}
In particular, by integrating along the outgoing light rays, we can conclude that $\lim_{r\rightarrow +\infty}r\phi(r+s,\omega)$ exists for any $(s,\omega)\in\mathbb{R}\times\mathbb{S}^{2}$. For any $M\geq 1$, we have the convergence rate estimate:
\begin{align*}
|\lim_{r\rightarrow+\infty}r\phi(r+s,r\omega)-\frac{M-s}{2}\phi(\frac{M+s}{2},\frac{M-s}{2}\omega)|\lesssim M^{-\frac{1}{2}},\  \forall\ |s|\leq M.
\end{align*}
\end{lemma}

\begin{proof}
In view of identity $\partial_{i}=\omega_{i}\partial_{r}+(\partial_{i}-\omega_{i}\partial_{r})$, we have
\begin{align*}
\sum_{i\not=j}\gamma^{i}\omega^{j}(x_{j}\partial_{i}-x_{i}\partial_{j}-\frac{1}{2}\gamma_{j}\gamma_{i})=\gamma^{i}r(\partial_{i}-\omega_{i}\partial_{r})-\gamma^{i}\omega_{i},
\end{align*}
therefore,
\begin{align*}
(\partial_{t}+\partial_{r})(rP(\omega)\phi)&=-\gamma^{0}\gamma^{i}r(\partial_{i}-\omega_{i}\partial_{r})\phi+P(\omega)\phi+\mathcal{D}\phi,\\
&=-\sum_{i\not= j}\gamma^{0}\gamma^{i}\omega^{j}(x_{j}\partial_{i}-x_{i}\partial_{j}-\frac{1}{2}\gamma_{j}\gamma_{i})\phi,\\&
+P(-\omega)\phi+\mathcal{D}\phi,
\end{align*}
multiply $P(\omega)$ on both sides, the LHS equals
\begin{align*}
(\partial_{t}+\partial_{r})(rP(\omega)\phi)&=-\sum_{i\not=j}\gamma^{0}\gamma^{i}P(-\omega)\omega^{j}(x_{j}\partial_{i}-x_{i}\partial_{j}-\frac{1}{2}\gamma_{j}\gamma_{i})\phi\\
&-\sum_{i\not=j}\omega^{i}\omega^{j}(x_{j}\partial_{i}-x_{i}\partial_{j}-\frac{1}{2}\gamma_{j}\gamma_{i})\phi+P(\omega)\mathcal{D}\phi,\\
&=-\sum_{i\not=j}\gamma^{0}\omega^{j}\gamma^{i}P(-\omega)(x_{j}\partial_{i}-x_{i}\partial_{j}-\frac{1}{2}\gamma_{j}\gamma_{i})\phi\\
&+P(\omega)\mathcal{D}\Phi.
\end{align*}
In view of Eq. \eqref{dirac: 222}, we have improved decay of
\begin{align*}
|P(-\omega)(x_{j}\partial_{i}-x_{i}\partial_{j}-\frac{1}{2}\gamma_{j}\gamma_{i})\phi|\lesssim \frac{\|\phi\|_{L^{\infty}(\mathbb{R}^{+},\mathcal{X})}}{(1+t+|x|)^{\frac{3}{2}}},
\end{align*}
therefore, we have
\begin{align*}
|(\partial_{t}+\partial_{r})(rP(\omega)\phi)|\lesssim  (1+t+|x|)^{-\frac{3}{2}}.
\end{align*}
\end{proof}

We use approximation method to prove the proposition \ref{dirac: the same definition}.
\begin{lemma}
For any $A\geq 1$, we have
\begin{align*}
\big\|\lim_{r\rightarrow+\infty}r\phi(r+s,r\omega)-\mathcal{F}_{nonlinear}^{+}(\phi_{0})(s,\omega)\big\|_{L^{\infty}([-A,A]\times\mathbb{S}^{2})}=0.
\end{align*}
In particular, we have
\begin{align*}
\lim_{r\rightarrow +\infty}r\phi(r+s,\omega)=\mathcal{F}^{+}_{nonlinear}(\phi_{0})(s,\omega),
\end{align*}
for almost every $(s,\omega)\in\mathbb{R}\times\mathbb{S}^{2}$.
\end{lemma}

\begin{proof}
Let $\chi\in C_{0}^{\infty}$ be a cut-off function such that $\chi\equiv 1$ on $[-1,1]$ and $\chi\equiv0$ on $(-\infty,2]\cup[2,+\infty)$,take
\begin{align*}
\Phi_{n}:=\chi(\frac{|x|}{n})\chi(\frac{t}{n})N(\phi,\phi),\phi_{0,n}:=\chi(\frac{|x|}{n})\phi_{0}.
\end{align*}
Let $\phi_{n}$ be solution of the equation
\begin{equation*}
\begin{cases}
\mathcal{D}\phi_{n}=\Phi_{n},\\
\phi_{n}(0,\cdot)=\phi_{0,n}.
\end{cases}
\end{equation*}
When $n\rightarrow+\infty$, we have
\begin{align*}
\big\|\mathcal{F}^{+}_{Dirac}(\phi_{0,n},\Phi_{n})-\mathcal{F}_{nonlinear}^{+}(\phi_{0})\big\|_{\mathcal{Y}^{+}}\rightarrow 0.
\end{align*}
By Sobolev inequality on $\mathbb{R}\times\mathbb{S}^{2}$,
\begin{align}\label{dirac: 104}
\big\|\mathcal{F}^{+}(\phi_{0,n},\Phi_{n})-\mathcal{F}^{+}_{nonlinear}(\phi_{0})\big\|_{L^{\infty}([-A,A]\times\mathbb{S}^{2})}\rightarrow 0.
\end{align}
In view of lemma \ref{dirac: 123}, for any $M\geq A$, we have
\begin{align*}
\big|\mathcal{F}^{+}_{Dirac}(\phi_{0,n},\Phi_{n})(s,\omega)-\frac{M-s}{2}\phi_{n}(\frac{M+s}{2},\frac{M-s}{2}\omega)\big|\lesssim M^{-\frac{1}{2}},\forall s\in [-A,A],\\
\big|\lim_{r\rightarrow +\infty}r\phi(r+s,\omega)(s,\omega)-\frac{M-s}{2}\phi(\frac{M+s}{2},\frac{M-s}{2}\omega\big|\lesssim M^{-\frac{1}{2}},\forall s\in [-A,A],
\end{align*}
by domain of dependence argument, when $n\geq M$ we can conclude that
\begin{align*}
\phi_{n}(\frac{M+s}{2},\frac{M-s}{2}\omega)=\phi(\frac{M+s}{2},\frac{M-s}{2}\omega), s\in [-A,A], 
\end{align*}
therefore, we have
\begin{align}\label{dirac: 103}
\lim_{n\rightarrow \infty}\big\|\lim_{r\rightarrow+\infty}r\phi(r+s,r\omega)-\mathcal{F}^{+}_{Dirac}(\phi_{0,n},\Phi_{n})(s,\omega)\big\|_{L^{\infty}([-A,A]\times\mathbb{S}^{2})}=0,
\end{align}
in view of Eq. \eqref{dirac: 104} and Eq. \eqref{dirac: 103}, we have
\begin{align*}
\big\|\lim_{r\rightarrow+\infty}r\phi(r+s,r\omega)-\mathcal{R}_{nonlinear}(\phi_{0})(s,\omega)\big\|_{L^{\infty}([-A,A]\times\mathbb{S}^{2})}=0.
\end{align*}
In particular, we have
\begin{align*}
\lim_{r\rightarrow +\infty}r\phi(r+s,\omega)=\mathcal{F}^{+}_{nonlinear}(\phi_{0})(s,\omega),
\end{align*}
for almost every $(s,\omega)\in\mathbb{R}\times\mathbb{S}^{2}$.
\end{proof}

\section*{Acknowledgments}
We sincerely thank the reviewer for his/her patience in pointing out the issues in the manuscript and for providing many constructive suggestions. The second author is funded by Guangzhou Science and Technology Programme (No.202201011125).


\end{document}